\documentclass[letterpaper, 12pt]{amsart}

\usepackage[margin=1in]{geometry}
\usepackage{amscd}
\usepackage{amsmath}
\usepackage{amsfonts}
\usepackage{amssymb}
\usepackage{amsthm}
\usepackage[backrefs, alphabetic]{amsrefs}
\usepackage{arydshln}
\usepackage{fancyhdr}
\usepackage{verbatim}
\usepackage[all]{xy}
\usepackage{color}
\usepackage[colorlinks=true, linkcolor=blue, citecolor=blue, pagebackref=true]{hyperref}
\usepackage[pagebackref=true]{hyperref}
\usepackage{mathtools}
\usepackage{etoolbox}
\patchcmd{\section}{\scshape}{\bfseries}{}{}
\makeatletter
\renewcommand{\@secnumfont}{\bfseries}
\makeatother
\newcommand{\RR}{\mathbb{R}}

\providecommand{\abs}[1]{\lvert#1\rvert}
\providecommand{\Abs}[1]{\left|#1\right|}
\providecommand{\norm}[1]{\lVert#1\rVert}
\providecommand{\Norm}[1]{\left\|#1\right\|}

\DeclarePairedDelimiter{\floor}{\lfloor}{\rfloor}

\def\eps{{\varepsilon}}
\def\en{{\omega}}
\def\al{{\alpha}}

\title{Khintchine types of translated coordinate hyperplanes}
\author[F.~A.~Ram{\'i}rez]{Felipe A.~Ram{\'i}rez}
\address{Department of Mathematics, University of York, UK}
\email{felipe.ramirez@york.ac.uk}
\date{}

\begin{document}

% ===============================================

\begin{abstract}
  There has been great interest in developing a theory of ``Khintchine
  types'' for manifolds embedded in Euclidean space, and considerable
  progress has been made for \emph{curved} manifolds. We treat the
  case of translates of coordinate hyperplanes, decidedly \emph{flat}
  manifolds. In our main results, we fix the value of one coordinate
  in Euclidean space and describe the set of points in the
  fiber over that fixed coordinate that are rationally approximable at
  a given rate. We identify translated coordinate hyperplanes for
  which there is a dichotomy as in Khintchine's Theorem: the set of
  rationally approximable points is null or full, according to the
  convergence or divergence of the series associated to the desired
  rate of approximation.
\end{abstract}

% ===============================================

\maketitle
\thispagestyle{empty}

\setcounter{tocdepth}{1} %excludes subsections from toc
\tableofcontents
% {\footnotesize{\tableofcontents}}

\section{Introduction}

\subsection{General setting}
The central object of study in simultaneous metric Diophantine
approximation is the set
\begin{equation*}
  \mathcal{W}_d (\psi) = \left\{\begin{split} &\boldsymbol{x}\in\mathbb{R}^d \textrm{ such that the inequality } \Norm{q\boldsymbol{x} - \boldsymbol{p}}_\infty < \psi(q) \\ &\textrm{holds for infinitely many } (\boldsymbol{p},q)\in \mathbb{Z}^d \times\mathbb{N}\end{split}\right\}
\end{equation*}
of \emph{$\psi$-approximable vectors} in $\mathbb{R}^d$, where $\psi:
\mathbb{N} \to \mathbb{R}^+ \cup \{0\}$ is a given map, which we call
an \emph{approximating function} if it is non-increasing. In words,
$\mathcal{W}_d(\psi)$ is the set of $d$-tuples of real numbers that
can be rationally approximated simultaneously, meaning with common
denominator, at the ``rate'' given by $\psi$, with infinitely many
different denominators.  For $\tau \in\mathbb{R}^+$ we denote
$\mathcal{W}_d(q\mapsto q^{-\tau}) = \mathcal{W}_d(\tau)$. The
supremum over all $\tau \in \mathbb{R}^+\cup\{\infty\}$ such that
$\boldsymbol{x}\in\mathcal{W}_d(\tau)$ is called the \emph{Diophantine
  type of $\boldsymbol{x}$}, and if it is $\infty$, then
$\boldsymbol{x}$ is called \emph{Liouville}. The Liouville numbers
form a set of Hausdorff dimension $0$ in $\mathbb{R}$.

\subsection{Foundational results}
The seminal result of Diophantine approximation, \emph{Dirichlet's
  Theorem} (c.\,1840), guarantees that if $\psi(q) \geq q^{-1/d}$,
then the set $\mathcal{W}_d (\psi)$ is all of $\mathbb{R}^d$. On the
other hand, a standard argument using the Borel--Cantelli Lemma shows
that if $\psi(q) \leq q^{-1/d - \eps}$ for some $\eps>0$, then
$m_d(\mathcal{W}_d(\psi))=0$, where $m_d$ is Lebesgue measure on
$\mathbb{R}^d$. One may guess that the difference lies in the
convergence or divergence of the integral of $\psi^d$. Indeed,
\emph{Khintchine's Theorem} (1926) set the foundation for simultaneous
metric Diophantine approximation by making this dichotomy
precise~\cite{Khintchine}.

\theoremstyle{plain} \newtheorem{theorem}{Theorem}

\theoremstyle{plain} \newtheorem*{kt}{Khintchine's Theorem}
\begin{kt}[1926]
  Let $\psi$ be an approximating function, and $d\in\mathbb{N}$. Then
  \[
  m_d(\mathcal{W}_d (\psi)) = \begin{dcases} \textsc{null} &
    \textrm{if}\quad \sum_{q=1}^\infty \psi(q)^d < \infty \\
    \textsc{full} & \textrm{if}\quad \sum_{q=1}^\infty \psi(q)^d =
    \infty.\end{dcases}
  \]
\end{kt}

\theoremstyle{remark} \newtheorem{remark}[theorem]{Remark}

\theoremstyle{remark} \newtheorem*{remark*}{Remark}
% \begin{remark*}
%   The convergence part of Khintchine's Theorem follows from the same
%   Borel--Cantelli argument mentioned above, so one should see the
%   divergence part as the real achievement.
% \end{remark*}

When it comes to the Lebesgue measure of $\mathcal{W}_d(\psi)$,
Khintchine's Theorem tells us the whole story. Of course, there are
other measures, and notions of size, that one may
consider. \emph{Jarn{\'i}k's Theorem} (1931) provides a similar
dichotomy for Hausdorff measures of $\mathcal{W}_d(\psi)$.

Later, Gallagher~\cite{Gallagherkt} extended Khintchine's Theorem in
the following sense.  \theoremstyle{plain}
\newtheorem*{gallagher}{Gallagher's Theorem}
\begin{gallagher}[1965]
  If $d\geq 2$, then Khintchine's Theorem is also true for functions
  $\psi: \mathbb{N}\to\mathbb{R}^+\cup\{0\}$ that are not monotone.
\end{gallagher}

\begin{remark*}
  Gallagher's Theorem is one of the main tools here. We use it in the
  proofs of
  Theorems~\ref{thm:prototype},~\ref{thm:prototypes},~\ref{thm:divergence},
  and~\ref{thm:upper}. (See~\S\ref{sec:onproofs}.)
\end{remark*}

\subsection{Current directions}
One of the major trends is in developing the theory of rational
approximations and ``Khintchine types'' for manifolds embedded in
$\mathbb{R}^d$. A manifold $\mathcal{M}\subset\mathbb{R}^d$ is said to
be of \emph{Khintchine type for divergence} if whenever $\psi$ is an
approximating function such that $\sum_{q\in\mathbb{N}}\psi(q)^d$
diverges, almost every point on $\mathcal{M}$ is
$\psi$-approximable. On the other hand, it is said to be of
\emph{Khintchine type for convergence} if whenever $\psi$ is an
approximating function such that $\sum_{q\in\mathbb{N}}\psi(q)^d$
converges, almost no point on $\mathcal{M}$ is $\psi$-approximable. If
it is both, it is of \emph{Khintchine type}.

Recently, Beresnevich, Dickinson, and Velani have shown that any
non-degenerate (meaning curved enough that no part of it is contained
in any hyperplane) submanifold of $\mathbb{R}^d$ is of Khintchine type
for divergence~\cite{BDVplanarcurves, Bermanifolds}. Vaughan and
Velani showed that non-degenerate planar curves are of Khintchine type
for convergence~\cite{VV06}.

\subsection{Our focus} This article is about the \emph{degenerate}
case. Far from deviating from all hyperplanes, the manifolds we
consider here \emph{are} hyperplanes. Specifically, we investigate
questions related to the following general problem:
\begin{quote}
  \emph{Describe the set of rationally approximable points in the
    fiber over a given fixed coordinate in Euclidean space.}
\end{quote}
For instance, suppose $\psi$ is an approximating function such that
$\sum\psi(q)^d$ diverges, say $\psi(q) = (q\log q)^{-1/d}$. Let
$x\in\mathbb{R}$ be fixed. In this example, Dirichlet's Theorem
guarantees that $x$ is $\psi$-approximable. But our $\psi$ decays
quite slowly, so we may expect that almost every point $(x,x_2, x_3,
\dots, x_n)\in\mathbb{R}^d$ in the fiber over $x$ is also
$\psi$-approximable. Our first result, Theorem~\ref{thm:prototype},
confirms this for $d\geq 3$. On the other hand, if we had chosen an
approximating function such that $\sum\psi(q)^d$ converges, then it
would make sense to seek the opposite statement: almost no points
$(x,x_2, x_3, \dots, x_d)\in\mathbb{R}^d$ are $\psi$-approximable,
with $x\in\mathbb{R}$ fixed. We find in Theorem~\ref{thm:convergence}
that this is sometimes true, sometimes not.

All of our results (presented in~\S\ref{sec:divresults}) are of a
similar flavor. Namely, they are steps toward the more general and
distant goal of bringing the theory of Khintchine types to the setting
of affine subspaces in $\RR^n$. Ultimately, one would like to be able
state a condition on an approximating function $\psi$ that is
equivalent to almost all points on a subspace being
$\psi$-approximable. As it stands, we only manage this for certain
hyperplanes (see Theorem~(b)). The rest of our results are sufficient
conditions for the ``almost all'' or ``almost no'' cases.

\section{Results}\label{sec:divresults}
\subsection{Divergence results for prototypical approximating
  functions}\label{sec:prototypes}

We have a number of results for the divergence situation, which for
illustrative purposes we state in order of increasing generality of
approximating functions. The first holds for the approximating
function $\psi(q)=(q\log q)^{-1/d}$.

\begin{theorem}\label{thm:prototype}
  Let $d\geq 3$ and $\psi(q) = (q\log q)^{-1/d}$. Then
  \[
  m_{d-1}\left(\mathcal{W}_{d}(\psi)\cap
    \left(\{x\}\times\mathbb{R}^{d-1}\right)\right)=\textsc{full}
  \]
  for every $x\in\mathbb{R}$.
\end{theorem}

From Theorem~\ref{thm:prototype} we can immediately deduce that the
same statement must hold for $\psi(q) = (q\log\dots\log q)^{-1/d}$,
because this function dominates $(q\log q)^{-1/d}$. Slightly more
challenging are approximating functions of the form
\[
\psi_{s,d}(q) = \left(\frac{1}{q (\log q)(\log\log
    q)\dots(\log\dots\log q)}\right)^{1/d}
\]
where $s\in\mathbb{N}$ is the length of the last string of
logarithms. For these we are able to prove the following.

\begin{theorem}\label{thm:prototypes}
  Let $d\geq 3$ and $s\in\mathbb{N}$. Then
  \[
  m_{d-1}\left(\mathcal{W}_{d}(\psi_{s,d})\cap
    \left(\{x\}\times\mathbb{R}^{d-1}\right)\right)=\textsc{full}
  \]
  for any $x$ whose Diophantine type is greater than $d$, and any $x$
  whose \emph{regular} Diophantine type is greater than $1$.
\end{theorem}

The regular Diophantine type of $x\in\mathbb{R}$ is the supremal
$\tilde\sigma\in[1,\infty)$ such that rational approximations
$\Abs{x-p/q}<q^{-(1+\tilde\sigma)}$ appear with positive lower
asymptotic density in the sequence $\{q_n\}_{n\geq 0}$ of continuants
of $x$. In simpler words, the regular Diophantine type of a number is
the maximal rate at which it can be rationally approximated, not just
infinitely often, but also with some frequency.

\begin{remark*}[On Khintchine's transference principle]
  We will present a proof of Theorem~\ref{thm:prototype} that holds
  for all non-Liouville $x$, and a proof of
  Theorem~\ref{thm:prototypes} that holds for non-Liouville $x$ with
  regular Diophantine type greater than $1$. The remaining cases are
  covered by Khintchine's transference principle, which implies that
  if $x\in\RR$ has Diophantine type greater than $d$, then
  \emph{every} point on $\{x\}\times\RR^{d-1}$ has Diophantine type
  greater than $1/d$. In particular, if $\psi$ is an approximating
  function that eventually dominates $q^{-(1+\eps)/d}$ for every
  $\eps>0$, then every point on $\{x\}\times\RR^{d-1}$ is
  $\psi$-approximable.
\end{remark*}

\begin{remark*}
  Theorem~\ref{thm:prototypes} is actually a corollary of a more
  general theorem (Theorem~\ref{thm:prototypesgen}) that holds for
  more fibers, but has a more technical statement. Both theorems are
  still true for uncountably many numbers \emph{not} satisfying their
  assumptions, including uncountably many numbers of \emph{any}
  Diophantine type and regular Diophantine type $1$, and \emph{every}
  number of Diophantine type at most the golden ratio regardless of
  regular Diophantine type. Such fibers are accounted for in
  Theorem~\ref{thm:divergence} below.
\end{remark*}

\subsection{Divergence result for approximating functions satisfying
  divergence condition}
In the next theorem we name fibers on which the desired ``almost
everywhere'' assertion can be made, provided only that the
approximating function $\psi$ is such that $\sum\psi(q)^d$
diverges. Among these are fibers over base points of Diophantine type
less than the golden ratio, or with an additional restriction, two,
and an uncountable set of fibers over base points of any given
Diophantine type.

\begin{theorem}\label{thm:divergence}
  Let $d\geq 3$. If $\psi$ is an approximating function such that the
  sum $\sum_{q\in\mathbb{N}} \psi(q)^d$ diverges, then
  \[
  m_{d-1}\left(\mathcal{W}_{d}(\psi)\cap
    \left(\{x\}\times\mathbb{R}^{d-1}\right)\right)=\textsc{full}
  \]
  for:
  \begin{itemize}
  \item[\rm\textbf{(a)}] Any $x \in\mathbb{Q}$ (even if $d=2$).
  \item[\rm\textbf{(b)}] Any $x\in\mathbb{R}\backslash\mathbb{Q}$ with
    the positive density property (see Definition~\ref{def:pdp}),
    including but not restricted to:
    \begin{itemize}
    \item Any $x\notin\mathcal W_1(\varphi)$ where
      $\varphi=\frac{1+\sqrt{5}}{2}$.
    \item Any $x\notin\mathcal W_1(2)$ for which there exists $R\geq
      1$ such that eventually whenever a partial quotient of $x$
      exceeds $R$, its continuants at least double before the next
      partial quotient exceeding $R$.
    \end{itemize}
  \item[\rm\textbf{(c)}] Uncountably many numbers of any Diophantine
    type.
  \end{itemize}
\end{theorem}

\begin{remark*}
  The sub-points in part~(b) come from Proposition~\ref{prop:pdp}.
\end{remark*}

One may ask whether Theorem~\ref{thm:divergence} holds for
non-monotonic functions. A simple observation shows that it cannot:
after fixing $x\in\mathbb{R}\backslash\mathbb{Q}$, consider the
function $\psi(q) = \norm{qx}$, where $\norm{\cdot}$ denotes distance
to the nearest integer. Then $\sum\psi(q)^d$ diverges, yet we can
never have $\norm{qx}<\psi(q)$, so the entire fiber over $x$ is
missing from $\mathcal{W}_d(\psi)$.

\subsection{Divergence result for approximating functions all of whose convergent subseries have zero density}
As to the question of whether the result of
Theorem~\ref{thm:divergence} holds for fibers other than those fitting
into parts (a), (b), or (c), we have the following theorem, which
gives a sufficient condition on the approximating function $\psi$ for
the result to hold on \emph{all} fibers. Recall that the density
$\operatorname{d}(A)$ of a set $A\subseteq\mathbb{N}$ is given by the
limit
\[
\operatorname{d}(A) = \lim_{N\to\infty}\frac{\Abs{A\cap[1, N]}}{N}
\]
when it exists.

\begin{theorem}\label{thm:upper}
  Let $d\geq 3$. If $\psi$ is an approximating function such that
  every convergent subseries $\sum_{q\in A} \psi(q)^d$ has asymptotic
  density $\operatorname{d}(A)=0$, then
  \[
  m_{d-1}\left(\mathcal{W}_{d}(\psi)\cap
    \left(\{x\}\times\mathbb{R}^{d-1}\right)\right)=\textsc{full}
  \]
  for all $x\in\mathbb{R}$.
\end{theorem}

For example, the approximating function $\psi(q) = cq^{-1/d}$, where
$c>0$, satisfies the requirement that all convergent subseries of
$\sum \psi(q)^d$ have asymptotic density $0$. Therefore, almost every
point on every $d-1$ dimensional fiber of $\mathbb{R}^d$ is
$\psi$-approximable. Of course, in the case $c=1$ we already knew this
(and \emph{more}) from Dirichlet's Theorem. But when we allow any
$c\in(0,1)$, Theorem~\ref{thm:upper} reflects the fact that
\emph{badly approximable vectors}---vectors
$\boldsymbol{x}\in\mathbb{R}^d$ for which there exists
$c:=c(\boldsymbol{x})>0$ such that $\norm{q\boldsymbol{x} -
  \boldsymbol{p}}_\infty \geq c q^{-1/d}$ for all
$(\boldsymbol{p},q)\in\mathbb{Z}^{d}\times\mathbb{N}$---do not
overpopulate any hyperplanes.

\theoremstyle{plain} \newtheorem{corollary}[theorem]{Corollary}

\subsection{Convergence result}
The next result deals with the convergence situation. Given an
approximating function such that $\sum_{q\in\mathbb{N}}\psi(q)^d$
converges, we would like to assert that \emph{almost no} points on the
fiber $\{x\}\times\mathbb{R}^{d-1}$ are $\psi$-approximable. Again, we
are able to make the desired statement for certain fibers, but not for
others, depending on the Diophantine type of the base-point.

\begin{theorem}\label{thm:convergence}
  Let $d\geq 2$. If $\psi$ is an approximating function such that the
  sum $\sum_{q\in\mathbb{N}} \psi(q)^d$ converges, then
  \[
  m_{d-1}\left(\mathcal{W}_{d}(\psi)\cap
    \left(\{x\}\times\mathbb{R}^{d-1}\right)\right)=\textsc{null}
  \]
  for:
  \begin{itemize}
  \item[\rm\textbf{(a)}] $\begin{dcases} \textrm{No $x \in\mathbb{Q}$}
      &\textrm{if $\sum_{q\in\mathbb{N}}\psi(q)^{d-1}$ diverges.} \\
      \textrm{Every $x\in\mathbb{R}$} &\textrm{if it
        converges.}\end{dcases}$
  \item[\rm\textbf{(b)}] Any $x\in\mathbb{R}\backslash\mathbb{Q}$ with
    the bounded ratio property (see Definition~\ref{def:brp}),
    including but not restricted to:
    \begin{itemize}
    \item Any $x$ of Diophantine type less than
      $\varphi=\frac{1+\sqrt{5}}{2}$.
    \item Any $x$ of Diophantine type less than $2$ for which there
      exists $R\geq 1$ such that eventually whenever a partial
      quotient of $x$ exceeds $R$, its continuants at least double
      before the next partial quotient exceeding $R$.
    \end{itemize}
  \end{itemize}
\end{theorem}

\begin{remark*}
  The subpoints in (b) are Proposition~\ref{prop:brp}. In part~(a), in
  the case that $\sum_{q\in\mathbb{N}}\psi(q)^{d-1}$ diverges, we get
  $\textsc{full}$ instead of $\textsc{null}$.
\end{remark*}

We were unaware during submission of this manuscript that
Theorem~\ref{thm:convergence}(b) actually follows
from~\cite[Theorem~1.6]{Ghosh} in the work of A.~Ghosh.\footnote{We
  thank the reviewer for bringing this paper to our attention.} He
describes ``dual'' approximability properties of points on hyperplanes
when the approximating function gives a convergent series. After
applying Khintchine's transference principle, one finds that Ghosh's
result implies in particular that coordinate hyperplanes in $\RR^d$,
translated perpendicularly by a distance of Diophantine type $< d$,
are of Khintchine type for convergence.

His methods come from dynamics on homogeneous spaces. Specifically,
the approximability properties of a point in $\RR^d$ are related to
the behavior of an associated flow orbit in the space of unimodular
lattices in $\RR^{d+1}$. Whether the orbit diverges into the cusp, and
at what rate, determines the Diophantine type of the point in $\RR^d$
(see~\cite{KM98}). Ghosh's work comes from a growing family of results
exploiting the connections between homogeneous dynamics and
Diophantine approximation, and its most immediate ancestor is a
paper~\cite{Kle03} of Kleinbock on extremality of affine subspaces of
$\RR^d$, relevant in~\S\ref{sec:ext}.

Our arguments for proving Theorem~\ref{thm:convergence} are very
elementary by comparison.

\subsection{A repackaging in terms of Khintchine types} 
We can state Theorems~\ref{thm:divergence} and~\ref{thm:convergence}
more succinctly by using the terminology of Khintchine types.

In the following statements, ``perpendicular translate of coordinate
hyperplane'' means a coordinate hyperplane that has been translated by
a vector perpendicular to it.

\theoremstyle{plain} \newtheorem*{theorema*}{Theorem (a)}
\begin{theorema*}
  Perpendicular translates of coordinate hyperplanes in $\mathbb{R}^d$
  (where $d\geq 2$) by rational numbers are of Khintchine type for
  divergence, but not for convergence.
\end{theorema*}

\theoremstyle{plain} \newtheorem*{theoremb*}{Theorem (b)}
\begin{theoremb*}
  Perpendicular translates of coordinate hyperplanes in $\mathbb{R}^d$
  (where $d\geq 3$) by numbers with the bounded ratio property are of
  Khintchine type. In fact they are of Khintchine type for convergence
  even when $d=2$.
\end{theoremb*}

\theoremstyle{plain} \newtheorem*{theoremc*}{Theorem (c)}
\begin{theoremc*}
  Uncountably many perpendicular translates of coordinate hyperplanes
  in $\mathbb{R}^d$ (where $d\geq 3$) by numbers of any given
  Diophantine type are of Khintchine type for divergence among
  approximating functions dominating any given.
\end{theoremc*}

\subsection{Extremality corollaries}\label{sec:ext}
There is a weaker notion than Khintchine type for convergence, called
``extremality.'' A manifold $\mathcal{M}\subset\mathbb{R}^d$ is
\emph{extremal} if for every approximating function such that
$\psi(q)\leq q^{-(1+\delta)/d}$ for some $\delta>0$, almost no point
on $\mathcal{M}$ is $\psi$-approximable.

The idea of extremality dates back to a 1932 conjecture of Mahler,
that Veronese curves are extremal. These are curves of the form
\[
(x, x^2, x^3, \dots, x^d)\subset \mathbb{R}^d.
\]
Mahler's conjecture was settled by Sprind{\v z}uk in 1964
(see~\cite{Spr69}), and this led to a great deal of research into the
extremality of curves, and in general manifolds, embedded in
$\mathbb{R}^d$. In the 1980s Sprind{\v z}uk conjectured that any
non-degenerate analytic submanifold of $\mathbb{R}^d$ is extremal, and
this was eventually settled by Kleinbock and Margulis~\cite{KM98} in
1998, even without analyticity.

Theorem~\ref{thm:convergence} yields some corollaries for extremality
of certain translated hyperplanes (\emph{degenerate} manifolds). They
were already known (and can be read from~\cite[Theorem~1.3]{Kle03}),
but we list them for the sake of completeness.

The following corollary is an immediate consequence of
Theorem~\ref{thm:convergence}(b).  \theoremstyle{plain}
\newtheorem*{corollary*}{Corollary}
\begin{corollary}\label{cor}
  Perpendicular translates of coordinate hyperplanes in
  $\mathbb{R}^d$, $d\geq 2$, by numbers with the bounded ratio
  property are extremal.
\end{corollary}

From our proofs we will also be able to read the following two
corollaries, also listing translated coordinate hyperplanes that are
extremal, this time according to their Diophantine type.

\begin{corollary}\label{cor:phi}
  Any perpendicular translate of a coordinate hyperplane by a number
  of Diophantine type $\varphi = \frac{1+\sqrt{5}}{2}$ or less is
  extremal.
\end{corollary}

\begin{corollary}\label{cor:2}
  Any perpendicular translate of a coordinate hyperplane by a number
  of Diophantine type $2$ or less, for which there exists $R\geq 1$
  such that eventually whenever a partial quotient of $x$ exceeds $R$,
  its continuants at least double before the next partial quotient
  exceeding $R$, is extremal.
\end{corollary}

\begin{remark*}
  Notice that in these corollaries the bounds on Diophantine type are
  not strict, whereas in Theorem~\ref{thm:convergence} (or, really,
  Proposition~\ref{prop:brp}) they are.
\end{remark*}

\begin{remark*}
  As we mentioned above, these corollaries already follow from the
  work of Kleinbock, which tells us exactly which hyperplanes are
  extremal and which are not. In fact, even more is known. Notice that
  to say that a submanifold is extremal is to say that almost every
  point on it is of Diophantine type $1/d$. It turns out that even if
  a subspace is \emph{not} extremal, almost all of its points still
  share a common Diophantine type, as do almost all the points on any
  non-degenerate submanifold of that subspace (where non-degeneracy in
  this case is determined with respect to the subspace). Details of
  this, and formulas for these Diophantine types, can be found
  in~\cite{Kle, Zhang}.
\end{remark*}
% ==================== ==================== ====================
% ==================== ==================== ====================
% ==================== ====================

% ===========================================
% ===========================================

\subsection{On the proofs}\label{sec:onproofs}
Our strategy for
Theorems~\ref{thm:prototype},~\ref{thm:prototypes},~\ref{thm:divergence},~\ref{thm:upper},
and~\ref{thm:convergence} is to arrive at a point where we can apply
either Khintchine's Theorem or Gallagher's Theorem to a hyperplane in
$\mathbb{R}^d$.

Given an approximating function $\psi$ and a point $x\in\mathbb{R}$,
we define a new function
\[
\bar\psi(q) := \begin{cases} \psi (q) &\textrm{if } \norm{qx}< \psi (q) \\
  0 &\textrm{if not,} \end{cases}
\]
where $\norm{\cdot}$ denotes distance to the nearest integer, and we
examine the sum
\begin{equation}\label{eqn:div}
  \sum_{q=1}^\infty \bar\psi(q)^{d-1}.
\end{equation}
If $d-1\geq 2$, we can apply Gallagher's Theorem to the fiber
$\{x\}\times\mathbb{R}^{d-1}$ and the non-monotonic function
$\bar\psi$, to prove that
\[
m_{d-1}\left(\mathcal{W}_{d-1}(\bar\psi)\right) =
m_{d-1}\left(\mathcal{W}_{d}(\psi)\cap
  \left(\{x\}\times\mathbb{R}^{d-1}\right)\right)
\]
is either \textsc{null} or \textsc{full}, depending on
whether~\eqref{eqn:div} converges or diverges.

All of the effort in all of our ``divergence'' results is in proving
the divergence of~\eqref{eqn:div} in different scenarios. Our strategy
for doing this is centered around showing that the intersection of the
set
\begin{equation*}
  \mathcal{Q}(x,\psi) = \left\{q\in\mathbb{N} : \norm{qx}<\psi(q) \right\}
\end{equation*}
with an interval $[M,N]$ grows quickly and steadily as the length
$N-M$ grows. For this it is most natural to think in terms of circle
rotations. We develop an argument based on the Three Gaps
Theorem. (See~\S\ref{sec:threegaps}.)

% \subsection{On the assumption that $d\geq 3$.}

For our ``convergence'' results, we try to show that~\eqref{eqn:div}
converges. Here we do not even need Gallagher, as the monotonicity
condition in Khintchine's Theorem is really only relevant to the
divergence part. It is well-known that the convergence part is an easy
consequence of the Borel--Cantelli Lemma, and holds even when $\psi$
is not monotone. This is why Theorem~\ref{thm:convergence} holds for
$d\geq2$.

Finally, we point out that although we do need $d\geq 3$ in order to
apply Gallagher's Theorem in our divergence results, it is not the
\emph{only} reason we make the assumption. Lemma~\ref{lem:wlog}
in~\S\ref{sec:lemmas} also requires it.

% ===========================================
% ===========================================
% ===========================================
% ===========================================

% ============================================

% ===========================================

\section{Mathematical preliminaries}\label{sec:cftgt}

\subsection{Asymptotic notations}

We use the following notations:
\begin{itemize}
\item $\ll$ means ``less than or equal to a positive multiple of.''
\item $\asymp$ means ``$\ll$ and $\gg$.''
\item $<^*, =^*,$ and $\leq^*$ mean ``eventually less than,'' ``-equal
  to,'' or ``-less than or equal to,'' respectively.
\item $\lesssim$ means ``less than or asymptotically equal to.''
\item $\sim$ means ``$\lesssim$ and $\gtrsim$," \emph{i.e.}
  ``asymptotically equal to.''
\end{itemize}

\subsection{Continued fractions}\label{sec:cfs}
For an irrational number $x\in\mathbb{R}\backslash\mathbb{Q}$, let
\begin{equation*}
  x = [a_0; a_1, a_2, a_3, \dots] = a_0 + \frac{1}{\displaystyle a_1
    + \frac{1}{\displaystyle a_2
      + \frac{1}{\displaystyle a_3 + \frac{1}{\ddots}}}}
\end{equation*}
be the simple continued fraction expansion of $x$, let
$\{p_k/q_k\}_{k\in\mathbb{N}}$ be its convergents, and
$\eta_k=\abs{q_k x - p_k}$ the associated differences. The continuants
$\{q_k\}$ follow the recursion $q_k = a_k q_{k-1} + q_{k-2}$ and
therefore grow at least exponentially fast. Every $m\in \mathbb{N}$
has a unique representation as $m= r q_k + q_{k-1} +s$ where $1\leq r
\leq a_{k+1}$ and $0\leq s < q_k$.

We take this opportunity to introduce a notation that we use
throughout the paper. Given $x=[a_0; a_1, a_2,
\dots]\in\mathbb{R}\backslash\mathbb{Q}$ and a fixed number $R\geq 0$,
let
\[
\left\{k_m:= k_m^{x,R}\right\}_{m\geq 0}
\]
be the sequence of indices where $a_{k_m+1}>R$, starting with the
conventional $k_0=-1$. Let $\Delta k_m := k_{m+1} - k_m$.

We will use the following simple lemma.

\theoremstyle{plain} \newtheorem{lemma}[theorem]{Lemma}

\begin{lemma}\label{lem:fibonacci}
  Let $\{F(n)\}_{n\in\mathbb{N}}:=\{1, 1, 2, 3, 5, \dots\}$ be the
  Fibonacci sequence. Then
  \[
  q_{k + n} \geq F(n+1) \, q_{k}
  \]
  for all $k, n\in\mathbb{N}$.
\end{lemma}

\begin{proof}[\textbf{Proof}]
  By the recursive relations between continuants, we have
  \begin{align*}
    q_k &\geq q_{k-1} + q_{k-2} \\
    &\geq 2q_{k-2} + q_{k-3} \\
    &\geq 3q_{k-3} + 2q_{k-4} \\
    &\geq 5q_{k-4} + 3q_{k-5} \\
    &\geq 8q_{k-5} + 5q_{k-6} \\
    &\vdots \\
    &\geq F(n+1) q_{k-n} + F(n) q_{k-n-1}
  \end{align*}
  for any $n< k$, which implies the result.
\end{proof}

In general this lemma may not give a very strong bound. We only use
the particular case $q_{k_{m+1}} \geq F(\Delta k_m)\,q_{k_m+1}$. For
an upper bound we have Lemma~\ref{lem:qbound} below.

\subsection{Diophantine type and growth of continuants}

Recall that for $\sigma\in[1,\infty)$, we define
\[
\mathcal W_1(\sigma) = \left\{x\in\mathbb{R}: \Abs{x - \frac{p}{q} }<
  \frac{1}{q^{1+\sigma}}\textrm{ for infinitely many }
  (p,q)\in\mathbb{Z}\times\mathbb{N}\right\}.
\]
It is a standard fact that the convergents of
$x\in\mathbb{R}\backslash\mathbb{Q}$ satisfy
\[
\frac{1}{q_n (q_n + q_{n+1})} < \Abs{x - \frac{p_n}{q_n}} <
\frac{1}{q_n q_{n+1}},
\]
and therefore we have
\[
x\in\mathcal W_1(\sigma) \implies q_n^\sigma < 2 q_{n+1}\textrm{ for
  infinitely many $n$}
\]
and
\[
q_n^\sigma < q_{n+1}\textrm{ for infinitely many $n$} \implies
x\in\mathcal W_1(\sigma).
\]
In particular, the Diophantine type of $x$ is the supremum over
$\sigma\in[1,\infty)$ such that $q_n^\sigma < q_{n+1}$ for infinitely
many $n\in\mathbb{N}$.

Conversely, $x\notin\mathcal W_1(\sigma)$ implies that $q_{n+1} \leq^*
q_n^\sigma$. We may equivalently define the Diophantine type of $x$ as
the infimum over $\sigma\in[1,\infty)$ for which $q_n^\sigma \gg
q_{n+1}$ as $n\to\infty$.

\begin{lemma}\label{lem:qbound}
  If $x\notin\mathcal W_1(\sigma)$, then
  \[
  q_{k_m+1} \leq^* (R+1)^{\sigma\Delta k_{m -1} + \sigma^2 \Delta k_{m
      -2} + \dots + \sigma^m \Delta k_0} \leq (R+1)^{\sigma^m k_m}
  \]
  for any $R\geq 1$.
\end{lemma}

\begin{proof}[\textbf{Proof}]
  Since $x\notin\mathcal W_1(\sigma)$, we have
  \begin{multline*}
    q_{k_m+1} \leq^* q_{k_m}^\sigma \leq (R+1)^{\sigma\Delta k_{m -1}}q_{k_{m-1}+1}^\sigma \\
    \leq^* \dots \leq^* (R+1)^{\sigma\Delta k_{m -1} + \sigma^2 \Delta
      k_{m -2} + \dots + \sigma^m \Delta k_0} \leq (R+1)^{\sigma^m
      k_m},
  \end{multline*}
  as claimed.
\end{proof}

\subsection{Types of Diophantine types}\label{sec:essdiophtype}

A number $x\in\mathbb{R}$ belongs to the set $\mathcal W_1(\sigma)$ of
$\sigma$-approximable numbers if there are infinitely many rational
approximations to $x$ with denominator $q$ satisfying $\norm{qx} <
q^{- \sigma}$. In view of the approximating properties of convergents,
this can be expressed as
\[
\mathcal W_1 (\sigma) = \left\{x : q_n^\sigma < q_{n+1}\textrm{ for
    infinitely many } n \in \mathbb{N}\right\},
\]
where $\{q_n\}$ are the continuants of $x$. It is useful to refine
this definition further by making a distinction between numbers
$x\in\mathcal W_1(\sigma)$ for which these approximating $q$'s appear
often, and those for which the $q$'s appear seldom.

\theoremstyle{definition} \newtheorem*{exdef*}{Example/Definition}
\begin{exdef*}[Uniform Diophantine type]
  Perhaps the most natural way to define ``frequent approximability''
  is to require that eventually \emph{all} continuants satisfy the
  growth condition. We may call
  \[
  \mathcal W_1^{\mathrm{uni}}(\sigma) = \left\{x : q_n^\sigma <
    q_{n+1}\textrm{ for all sufficiently large }
    n\subseteq\mathbb{N}\right\}
  \]
  the set of \emph{uniformly $\sigma$-approximable} numbers. Notice
  that this means, in particular, that the set of continuants
  satisfying the growth condition has density $1$ as a subsequence of
  $\{q_n\}_{n\geq 0}$. The following definition relaxes this.
\end{exdef*}

\begin{exdef*}[Regular Diophantine type]
  Another natural notion of frequent approximability is captured by
  the set of \emph{regularly $\sigma$-approximable} numbers:
  \[
  \mathcal W_1^{\mathrm{reg}}(\sigma) = \left\{x : q_{n_j}^\sigma <
    q_{n_j+1}\textrm{ for some s.p.l.a.d. }
    \{n_j\}\subseteq\mathbb{N}\right\}
  \]
  where s.p.l.a.d. stands for ``sequence of positive lower asymptotic
  density.'' It is obvious that $\mathcal W_1^{\textrm{uni}}(\sigma)
  \subset \mathcal W_1^{\mathrm{reg}}(\sigma)\subset \mathcal
  W_1(\sigma)$. Notice that $\mathcal W_1^{\mathrm{reg}}(1) =
  \mathbb{R}$, because all continuants satisfy $q_n < q_{n+1}$. We
  define the \emph{regular Diophantine type} of $x$ to be the supremum
  over $\sigma\in[1,\infty)$ such that $x\in\mathcal
  W_1^{\mathrm{reg}}(\sigma)$.
\end{exdef*}

Actually, we will work with a more permissive set.

\begin{exdef*}[Essential Diophantine type]
  We define the set of \emph{essentially $\sigma$-approximable}
  numbers to be
  \begin{equation*}
    \mathcal W_1^{\mathrm{ess}}(\sigma) = \left\{\begin{split}&x\in\mathbb{R} \textrm{ such that there exists } R\geq 0 \textrm{ for which } \\&q_{k_{m_j}}^\sigma < q_{k_{m_j}+1} \textrm{ on some s.p.l.a.d. } \{m_j\}\subseteq\mathbb{N} \end{split}\right\}.
  \end{equation*}
  The containments $\mathcal W_1^{\textrm{uni}}(\sigma)\subset\mathcal
  W_1^{\mathrm{reg}}(\sigma) \subset \mathcal
  W_1^{\mathrm{ess}}(\sigma)\subset \mathcal W_1(\sigma)$ are
  clear. Again, any number $x$ is an element of $\mathcal
  W_1^{\mathrm{ess}}(1)$, and we define its \emph{essential
    Diophantine type} to be the supremum over $\sigma\in[1,\infty)$
  where $x\in\mathcal W_1^{\mathrm{ess}}(\sigma)$.
\end{exdef*}

\subsection{Three Gaps Theorem}\label{sec:threegaps}
For any $x\in\mathbb{R}$ and $m\in\mathbb{N}$ the set $\{q x +
\mathbb{Z}\}_{q=1}^m\subset\mathbb{R}/\mathbb{Z}$ cuts the circle
$\mathbb{R}/\mathbb{Z}$ into arcs of at most three different lengths;
this is known as the Three Gaps Theorem.

For $m\in\mathbb{N}$, write
\[
m= r q_k + q_{k-1} +s
\]
where $1\leq r \leq a_{k+1}$ and $0\leq s < q_k$ as
in~\S\ref{sec:cfs}, and let $r_x(m+1)$ denote the ratio of the longest
gap length to the shortest gap length in the trajectory~$\{q x +
\mathbb{Z}\}_{q=1}^{m+1} \subset\mathbb{R}/\mathbb{Z}$. Then for
$m\in\mathbb{N}$,
\begin{equation}\label{eqn:ratios}
  r_x(m+1) =\begin{dcases} \epsilon + \frac{\eta_{k+2}}{\eta_{k+1}} + a_{k+2} &\textrm{if } r=a_{k+1}\\ \epsilon + \frac{\eta_{k+1}}{\eta_{k}} + (a_{k+1}-r) &\textrm{if } r<a_{k+1} \end{dcases}
\end{equation}
where $\epsilon =1$ unless $s = q_k - 1$, in which case
$\epsilon=0$. (See~\cite{gapratio}.)

\section{Sequences with bounded gap ratios}\label{sec:brp}

Formula~\eqref{eqn:ratios} shows that $r_x$ is always bounded if and
only if $x$ is badly approximable. On the other hand, for any $R\geq
1$ it is easy to generate a sequence
\[
\left\{L_n:=L_n^R := L_n^{x,R}\right\}\subseteq\mathbb{N}
\]
such that the ratios $r_x(L_n)$ are bounded by $R$ for all $n$,
regardless of the continued fraction expansion of $x$. The reason we
would want to do this is so that we can control the density of points
on partial orbits of $x$ of \emph{length} $L_n$.

\begin{lemma}\label{lem:density}
  Let $x\in\mathbb{R}\backslash\mathbb{Q}$. Suppose the gap ratio for
  $\{qx+\mathbb{Z}\}_{q=1}^L\subset \mathbb{R}/\mathbb{Z}$ is bounded
  by $R$, and $L\geq 2$. Then for any $q_0\in\mathbb{N}$,
  \[
  \frac{1}{RL} < \ell_{\min} < \frac{1}{L} < \ell_{\max} < \frac{R}{L}
  \]
  where $\ell_{\min}$ and $\ell_{\max}$ are the minimum and maximum
  arc-lengths into which the set $\{qx +
  \mathbb{Z}\}_{q=q_0+1}^{q_0+L}$ cuts the circle.
\end{lemma}

\begin{proof}[\textbf{Proof}]
  The $L$ points of $\{qx + \mathbb{Z}\}_{q=1}^L$ partition the circle
  $\mathbb{R}/\mathbb{Z}$ into $L$ intervals. Let $\ell_{\min}$ and
  $\ell_{\max}$ be the shortest and longest lengths of these
  intervals. Assuming $x$ is irrational and $L\geq 2$, we have
  $\ell_{\min} < \frac{1}{L} < \ell_{\max}$. (Of course, if $L=1$,
  then $\ell_{\min} = \ell_{\max} = 1$, no matter what $x$ is.) By the
  ratio bound, $\ell_{\max} \leq R\ell_{\min}$. Putting the two
  inequalities together gives the desired system of inequalities,
  which is of course unchanged by a rotation by $q_0 x$.
\end{proof}

\def\ar{H}

\begin{lemma}\label{lem:brsequence}
  Let $R \geq 0$. We will have $r_x (L) \leq 2+R$ exactly when $L$
  belongs to some block
  \[
  \left\{q_k - R q_{k-1}, \dots, q_\ell\right\} \subseteq\mathbb{N}
  \]
  of consecutive integers, where $k=0$ or $a_k > R$, and $\ell \geq k$
  indexes the next time $a_{\ell+1}> R$ again. (If it never happens
  again, we interpret this as $\ell=\infty$ and $q_\infty = \infty$.)
\end{lemma}

\begin{proof}[\textbf{Proof}]
  This follows simply by consulting~\eqref{eqn:ratios}. We can list
  all of the numbers $m\in\mathbb{N}$ that will result in bounded gap
  ratios, and find that $r_x (m+1) \leq 2+R$ exactly when
  \[
  m\in \begin{dcases}
    \{q_k - Rq_{k-1} -1, \dots, q_k  -1\} &\textrm{for some } a_k > R \\
    \{q_k, \dots, q_{k+1} -1\} &\textrm{for some } a_{k+1} \leq R.
  \end{dcases}
  \]
  Concatenating these blocks and setting $L=m+1$ gives the lemma.
\end{proof}

\begin{remark*}
  A consequence of this lemma that is interesting in itself (and
  probably known already to experts) is that the continuants
  $\{q_n\}_{n=0}^\infty$ are \emph{exactly} the times when the gap
  ratios for $\{qx\}_{q = 1}^{q_n}$ are bounded by $2$.
\end{remark*}

Forming the sequence $\left\{k_m:=k_m^{x,R}\right\}$ and putting
\begin{equation}\label{eqn:Bm}
  B_{m+1} = B_{m+1}^{x,R} = \left[q_{k_m+1} - Rq_{k_m}, q_{k_{m+1}}\right]\cap\mathbb{N},
\end{equation}
Lemma~\ref{lem:brsequence} implies that our sequence of $2+R$-bounded
gap ratios is the concatenation $\{L_n\} = \{B_1, B_2, B_3, \dots
\}$. If the sequence $\{k_m\}$ terminates at $k_t$, then
\[
B_{t+1} = \left[q_{k_t+1} - R q_{k_t}, \infty \right)\cap\mathbb{N}.
\]
This happens only if $x\in\mathbb{R}\backslash\mathbb{Q}$ is a badly
approximable number, and conversely if $x$ is badly approximable, we
can choose $R\geq 0$ large enough that this happens.

\subsection{Calculations based on~\eqref{eqn:Bm}}

It will be useful to keep certain measurements of $B_m^{x,R}$ in
mind. First, the length of the block $B_m$ is
\begin{equation}\label{BmRlength}
  \Abs{B_{m+1}} = q_{k_{m+1}} - q_{k_m+1} + R q_{k_m}+1.
\end{equation}
If the sequence $\{k_m\}$ terminates at $k_t$, then we can obviously
consider $\Abs{B_{t+1}}$ to be infinite.

Let $\{\en_m\}_{m= 1}^{\infty}$ be the sequence such that $L_{\en_m} =
q_{k_m}$ is the right end-point of the block $B_m$. Then $\en_m$ is
the sum of the lengths of the blocks $B_1, \dots, B_m$, which,
by~\eqref{BmRlength} is
\begin{equation*}
  \en_m = \sum_{n=0}^{m-1} q_{k_{n+1}} - q_{k_n+1} + Rq_{k_n}+1.
\end{equation*}
Let $\al_m$ be the index for the left end-point $L_{\al_m}$ of the
block $B_m$, so that $\al_m = \en_{m-1}+1$ for all $m\in\mathbb{N}$,
and $\al_1=1$.

The distance between consecutive blocks $B_{m+1}$ and $B_m$ is
\begin{equation*}
  B_{m+1} - B_m:= \min B_{m+1} - \max B_m =q_{k_m+1} - (R+1) q_{k_m}.
\end{equation*}
The following lemma describes the sum $\Sigma_{B_m}$ of the elements
in block $B_m$.

\begin{lemma}\label{lem:BmRsum}
  We have
  \[
  \Sigma_{B_m} \sim\frac{1}{2}\left(q_{k_m}^2 - (q_{k_{m-1}+1} -
    Rq_{k_{m-1}})^2\right).
  \]
  In particular,
  \[
  q_{k_m} q_{k_{m-1}} \ll \Sigma_{B_m} \ll q_{k_m}q_{k_m -1}
  \]
  for all $m\in\mathbb{N}$. Also,
  \[
  q_{k_m}^2 \ll \Sigma_{B_m} \ll q_{k_m}q_{k_m -1}
  \]
  whenever $q_{k_m} \geq 2q_{k_{m-1}+1}$ (\emph{i.e.} if $a_{k_m}>1$
  or if $k_m - k_{m-1} > 2$).
\end{lemma}

\begin{proof}[\textbf{Proof}]
  Block sums are given by the formula
  \begin{multline*}
    \Sigma_{B_m} = \frac{1}{2}\left(q_{k_m}+ q_{k_{m-1}+1} - Rq_{k_{m-1}}\right)\left(q_{k_m}- q_{k_{m-1}+1} + Rq_{k_{m-1}} + 1\right) \\
    \sim \frac{1}{2}\left(q_{k_m}^2 - (q_{k_{m-1}+1} -
      Rq_{k_{m-1}})^2\right).
  \end{multline*}
  If $k_m \neq k_{m-1}+1$,
  \begin{align*}
    \frac{\Sigma_{B_m}}{q_{k_m}q_{k_m-1}} &\sim \frac{q_{k_m}^2 - (q_{k_{m-1}+1} - Rq_{k_{m-1}})^2}{2q_{k_m}q_{k_m-1}} \\
    &= \frac{1}{2}\left(\frac{q_{k_m}}{q_{k_m-1}}  -  \frac{q_{k_{m-1}+1}^2}{q_{k_m}q_{k_m-1}} - R^2\frac{q_{k_{m-1}}^2}{q_{k_m}q_{k_m-1}} + 2R \frac{q_{k_{m-1}+1}q_{k_{m-1}}}{q_{k_m}q_{k_m-1}}\right) \\
    &\leq \frac{1}{2}\left((a_{k_m}+1) + 2R\right) \ll 1
  \end{align*}
  because $a_{k_m}\leq R$ in this case. On the other hand, if $ k_m =
  k_{m-1} + 1$,
  \begin{align*}
    \frac{\Sigma_{B_m}}{q_{k_m}q_{k_m-1}} &\sim \frac{q_{k_m}^2 - (q_{k_m} - Rq_{k_{m-1}})^2}{2q_{k_m}q_{k_m-1}} \\
    &= \frac{2R q_{k_m}q_{k_{m-1}} -
      R^2q_{k_{m-1}}^2}{2q_{k_m}q_{k_m-1}} \ll 1,
  \end{align*}
  which establishes the upper bound.

  For the lower bound, first suppose that $k_{m-1}+1 = k_m$. In this
  case we have
  \[
  q_{k_m}q_{k_m-1} = \left(a_{k_m}q_{k_m-1} +
    q_{k_m-2}\right)q_{k_m-1} \geq \left(R+1\right)q_{k_m-1}^2,
  \]
  so that
  \begin{align*}
    \frac{q_{k_m}^2 - \left(q_{k_{m-1}+1} - Rq_{k_{m-1}}\right)^2}{q_{k_m}q_{k_m-1}} &= \frac{2Rq_{k_m}q_{k_m-1} - R^2 q_{k_m-1}^2}{q_{k_m}q_{k_m-1}}\\
    &= 2R - \frac{R^2 q_{k_m-1}^2}{q_{k_m}q_{k_m-1}} = 2R -
    \frac{R^2}{R+1} > R,
  \end{align*}
  proving $\Sigma_{B_m}\gg q_{k_m}q_{k_m-1} = q_{k_m}q_{k_{m-1}}$ in
  this case.

  If $k_{m-1}+2 = k_m$ and $a_{k_m}=1$ then
  \begin{multline*}
    \Sigma_{B_m} = \frac{1}{2}\left(q_{k_m}+ q_{k_{m-1}+1} - Rq_{k_{m-1}}\right)\left(q_{k_m}- q_{k_{m-1}+1} + Rq_{k_{m-1}} + 1\right) \\
    = \frac{1}{2}\left(2 q_{k_{m-1}+1} - (R-1)q_{k_{m-1}}\right)\left((R+1)q_{k_{m-1}} + 1\right) \\
    = (R+1)q_{k_{m-1}+1}q_{k_{m-1}} -
    \frac{1}{2}(R+1)(R-1)q_{k_{m-1}}^2 + q_{k_{m-1}+1} -
    \frac{1}{2}(R-1)q_{k_{m-1}}.
  \end{multline*}
  Dividing by $q_{k_{m-1}+1}q_{k_{m-1}}$ gives
  \begin{multline*}
    \frac{\Sigma_{B_m}}{q_{k_{m-1}+1}q_{k_{m-1}}}= (R+1) - \frac{(R+1)(R-1)q_{k_{m-1}}^2}{2 q_{k_{m-1}+1}q_{k_{m-1}}} + \frac{q_{k_{m-1}+1}}{q_{k_{m-1}+1}q_{k_{m-1}}} - \frac{(R-1)q_{k_{m-1}}}{2q_{k_{m-1}+1}q_{k_{m-1}}} \\
    \sim (R+1) - \frac{(R+1)(R-1)q_{k_{m-1}}}{2 q_{k_{m-1}+1}}\geq
    (R+1) - \frac{(R-1)}{2} \gg 1.
  \end{multline*}
  which proves $\gg q_{k_m-1}q_{k_m-2}$. But in this case we have
  $q_{k_m} = q_{k_m-1}+q_{k_m-2} \leq 2q_{k_m-1}$, so we have proved
  $\Sigma_{B_m} \gg q_{k_m}q_{k_m-2} = q_{k_m}q_{k_{m-1}}$ in this
  case.

  In the remaining cases we have $k_{m-1}+1 \neq k_m$ and there is
  some integer $A\geq 2$ such that
  \[
  q_{k_m} \geq A q_{k_{m-1}+1} + q_{k_{m-1}}.
  \]
  We write
  \[
  q_{k_m}^2 - \left(q_{k_{m-1}+1} - Rq_{k_{m-1}}\right)^2 =
  \left(\frac{A^2 - 1}{A^2}\right) q_{k_m}^2 + \frac{1}{A^2} q_{k_m}^2
  - \left(q_{k_{m-1}+1} - Rq_{k_{m-1}}\right)^2
  \]
  and proceed to bound
  \begin{multline*}
    \left(\frac{A^2 - 1}{A^2}\right) q_{k_m}^2 + \frac{1}{A^2} q_{k_m}^2 - \left(q_{k_{m-1}+1} - Rq_{k_{m-1}}\right)^2 \\
    \geq \left(\frac{A^2 - 1}{A^2}\right) q_{k_m}^2 + \left(q_{k_{m-1}+1} + \frac{1}{A} q_{k_{m-1}}\right)^2 \\
    - \left(q_{k_{m-1}+1} +\frac{1}{A^2} q_{k_{m-1}} - \left(R + \frac{1}{A}\right)q_{k_{m-1}}\right)^2 \\
    = \left(\frac{A^2 - 1}{A^2}\right) q_{k_m}^2 + 2\left(R+\frac{1}{A}\right)\left(q_{k_{m-1}+1} + \frac{1}{A} q_{k_{m-1}}\right)q_{k_{m-1}} - \left(R+\frac{1}{A}\right)^2q_{k_{m-1}}^2 \\
    \geq \left(\frac{A^2 - 1}{A^2}\right) q_{k_m}^2 +
    \left(R+\frac{1}{A}\right)^2q_{k_{m-1}}^2 \gg q_{k_m}^2
  \end{multline*}
  because $A\geq 2$.
\end{proof}

\subsection{Positive density property}
 
The following definition is relevant to our ``divergence'' results.
\theoremstyle{definition} \newtheorem{definition}[theorem]{Definition}
\begin{definition}[Positive density property]\label{def:pdp}
  We say $x\in\mathbb{R}\backslash\mathbb{Q}$ has the \emph{positive
    density property} if there exists $R\geq 1$ such that
  \[
  \limsup_{m\to\infty} \frac{L_{\al_m}^R}{\Sigma_{\al_m}^R} <1.
  \]
  An intuitive interpretation is that a number with positive density
  property has blocks $B_m:=B_m^{x,R}$ that are not too far away from
  each other.
\end{definition}

\theoremstyle{plain} \newtheorem{proposition}[theorem]{Proposition}
\begin{proposition}\label{prop:pdp}
  The number $x\in\mathbb{R}\backslash\mathbb{Q}$ has the positive
  density property if and only if
  \[
  q_{k_m+1} - Rq_{k_m} \ll \sum_{\ell=1}^m q_{k_\ell}^2 -
  (q_{k_{\ell-1}+1} - Rq_{k_{\ell-1}})^2
  \]
  as $m\to\infty$. In particular,
  \begin{itemize}
  \item any $x\notin\mathcal W_1 (\varphi)$, and
  \item any $x\notin\mathcal W_1 (2)$ for which there exists $R\geq 1$
    such that $q_{k_m} \geq 2q_{k_{m-1}+1}$ for all but finitely many
    $m\in\mathbb{N}$
  \end{itemize}
  has the positive density property.
\end{proposition}

\begin{proof}[\textbf{Proof}]
  Positive density property is the requirement that there is some
  $\delta < 1$ such that
  \[
  \frac{L_{\al_{m+1}}^R}{\Sigma_{\al_{m+1}}^R} = \frac{q_{k_m^R+1} -
    Rq_{k_m^R}}{q_{k_m^R+1} - Rq_{k_m^R} + \Sigma_{\en_m}^R} \leq
  \delta
  \]
  for all sufficiently large $m$. This is equivalent to $q_{k_m+1} -
  Rq_{k_m} \ll \Sigma_{\en_m}$, which by Lemma~\ref{lem:BmRsum} is
  equivalent to
  \begin{equation*}
    q_{k_m+1} - Rq_{k_m} \ll \sum_{\ell=1}^m q_{k_\ell}^2 - (q_{k_{\ell-1}+1} - Rq_{k_{\ell-1}})^2.
  \end{equation*}
  In particular,
  \begin{equation}\label{eqn:suff}
    q_{k_m+1} - Rq_{k_m} \ll \Sigma_{B_m}
  \end{equation}
  is sufficient.

  If $k_{m-1}+1 = k_m$, the sufficient~\eqref{eqn:suff} becomes
  \[
  q_{k_m+1} - Rq_{k_m} \ll q_{k_m}q_{k_m-1},
  \]
  for which it is sufficient that
  \[
  \frac{q_{k_m+1}}{q_{k_m}q_{k_m-1}}\ll 1.
  \]
  We will have this comparison whenever $x\notin \mathcal
  W_1(\varphi)$.

  On the other hand, if $k_{m-1}+1 \neq k_m$, then
  Lemma~\ref{lem:BmRsum} allows us to consider two cases: either
  $\Delta k_{m-1}:=k_m - k_{m-1} =2$ and $a_{k_m}=1$, or $q_{k_m}\geq
  2q_{k_{m-1}+1}$. In the first case,~\eqref{eqn:suff} becomes
  \[
  q_{k_m+1} - Rq_{k_m} \ll q_{k_m}q_{k_{m-1}} = q_{k_m}q_{k_m -2},
  \]
  and for this it is sufficient that
  \begin{equation*}\label{eqn:km2}
    \frac{q_{k_m+1}}{q_{k_m}q_{k_m -2}} \ll 1.
  \end{equation*}
  If $x\notin\mathcal W_1(\varphi)$, then
  \[
  \frac{q_{k_m+1}}{q_{k_m}q_{k_m -2}} \ll
  \frac{q_{k_m}^\varphi}{q_{k_m}q_{k_m - 1}^{1/\varphi}}
  \]
  and since $q_{k_m} = q_{k_m-1} + q_{k_m-2} \leq 2q_{k_m-1}$ in this
  case,
  \[
  \frac{q_{k_m}^\varphi}{q_{k_m}q_{k_m - 1}^{1/\varphi}} \ll
  \frac{q_{k_m}^\varphi}{q_{k_m}^{1 + 1/\varphi}} \ll 1
  \]
  as wanted.

  In the second case we will have $q_{k_m} \geq 2 q_{k_{m-1}+1}$, and
  the last part of Lemma~\ref{lem:BmRsum} implies that
  \[
  q_{k_m+1} - Rq_{k_m} \ll q_{k_m}^2
  \]
  is sufficient for~\eqref{eqn:suff}. This is satisfied whenever
  $x\notin\mathcal W_1(2)$. In particular, if $x\notin\mathcal W_1
  (\varphi)$, then we satisfy~\eqref{eqn:suff}, which proves the first
  point in the proposition. This last paragraph has also proved the
  second point in the proposition.
\end{proof}

\subsection{Bounded ratio property}

The following property is slightly stronger than positive density
property. It is relevant to our ``convergence'' results.

\begin{definition}[Bounded ratio property]\label{def:brp}
  We say that $x\in\mathbb{R}\backslash\mathbb{Q}$ has the
  \emph{bounded ratio property} if there exists a bound $R\geq1$ such
  that
  \[
  \sum_{m\in\mathbb{N}}\frac{B_{m+1}^R - B_m^R}{\Sigma_{\en_m}^R} <
  \infty.
  \]
  % (where we consider the first term $\frac{B_2^R -
  % B_1^R}{\Sigma_{\en_1}^R}$ to be empty if it happens that $B_1^R =
  % \emptyset$).
  This is equivalent to
  \[
  \sum_{m\in\mathbb{N}}\frac{L_{\al_{m+1}}}{\Sigma_{\en_m}} < \infty.
  \]
  Again, having the bounded ratio property means that the jumps
  between the blocks $B_m$ are not too severe.
\end{definition}

The following proposition gives numbers with the bounded ratio
property, based on Diophantine type.

\begin{proposition}\label{prop:brp}
  Numbers with bounded ratio property:
  \begin{itemize}
  \item Every number of Diophantine type less than
    $\varphi=\frac{1+\sqrt{5}}{2}$ has the bounded ratio property.
  \item Every number of Diophantine type less than $2$ for which there
    is some $R\geq 1$ such that $q_{k_m} \geq 2q_{k_{m-1}+1}$ for all
    but finitely many $m\in\mathbb{N}$ has the bounded ratio property.
  \end{itemize}
\end{proposition}

\begin{remark*}
  Notice that these are not the same numbers listed in
  Proposition~\ref{prop:pdp}. There, we require (for example) that
  $x\notin\mathcal W_1(\varphi)$, whereas here we are requiring that
  $x\notin\mathcal W_1(\sigma)$ for some $\sigma <\varphi$. This is a
  slightly stronger requirement.
\end{remark*}

\begin{proof}[\textbf{Proof}]
  For the first assertion, let $\sigma<\varphi$ be such that
  $x\notin\mathcal W_1 (\sigma)$. By Lemma~\ref{lem:BmRsum} we have
  \[
  \frac{L_{\al_{m+1}}}{\Sigma_{\en_m}}
  \ll \frac{q_{k_m+1}}{q_{k_m}q_{k_m-1}}
  \]
  as long as we are not in the situation where $\Delta k_{m-1} = 2$
  and $a_{k_m}=1$. This in turn is bounded
  \[
  \frac{L_{\al_{m+1}}}{\Sigma_{\en_m}}
  \ll \frac{q_{k_m+1}}{q_{k_m}q_{k_m-1}}
  \ll q_{k_m}^{\sigma-1-1/\sigma}.
  \]
  On the other hand, if we \emph{are} in the situation of $\Delta
  k_{m-1} = 2$ and $a_{k_m}=1$, then $q_{k_m}\asymp q_{k_m-1}$, so
  \[
  \frac{L_{\al_{m+1}}}{\Sigma_{\en_m}}
  \ll \frac{q_{k_m+1}}{q_{k_m}q_{k_m-2}}
  \ll \frac{q_{k_m}^\sigma}{q_{k_m-1}^{1 + 1/\sigma}}
  \ll q_{k_m}^{\sigma - 1 - 1/\sigma},
  \]
  as above. And the sum
  \[
  \sum_{m\in\mathbb{N}} q_{k_m}^{\sigma - 1 - 1/\sigma}
  \]
  converges because $\sigma - 1 - 1/\sigma < 0$. Therefore, $x$ has
  the bounded ratio property.

  For the second assertion, let $\sigma<2$ and let $x\notin \mathcal
  W_1(\sigma)$ be such that $q_{k_m} \geq 2q_{k_{m-1}+1}$ for all but
  finitely many $m\in\mathbb{N}$, for some $R\geq 1$. By
  Lemma~\ref{lem:BmRsum},
  \[
  \frac{L_{\al_{m+1}}}{\Sigma_{\en_m}}
  \ll \frac{q_{k_m+1}}{q_{k_m}^2} \ll q_{k_m}^{\sigma -2},
  \]
  and the sum
  \[
  \sum_{m\in\mathbb{N}} q_{k_m}^{\sigma - 2}
  \]
  diverges because $\sigma-2 <0$. Therefore $x$ has the bounded ratio
  property and the proposition is proved.
\end{proof}

\section{Some counting lemmas}\label{sec:lemmas}

This section is about the counting Lemmas~\ref{lem:schmidt}
and~\ref{lem:antischmidt}. They give bounds on
\[
\Abs{\mathcal Q(x,\psi)\cap[M, N]}
\]
when $M$ and $N$ come from our bounded ratio sequences $\{L_n\}$.

\begin{lemma}\label{lem:schmidt}
  Let $\{L_n\}$ be a sequence of $R$-bounded gap ratios for
  $x\in\mathbb{R}\backslash\mathbb{Q}$, and $\{\Sigma_n\}$ be the
  sequence defined by $\Sigma_n = L_1+L_2+\dots+L_n$. If $\psi$ is an
  approximating function such that $L_n\, \psi (\Sigma_n)\geq R$ for
  $n$ sufficiently large, then
  \[
  \sum_{\substack{q=\Sigma_n+1\\ q \in
      \mathcal{Q}(x,\psi)}}^{\Sigma_{n+1}} 1 \gg L_{n+1}\, \psi
  (\Sigma_{n+1})\quad\textrm{as $n\to\infty$,}
  \]
  where
  \[
  \mathcal{Q}(x, \psi) = \left\{q\in\mathbb{N} :
    \norm{qx}<\psi(q)\right\}
  \]
  is the set of denominators that $\psi$-approximate $x$ in
  $\mathbb{R}$.
\end{lemma}

\begin{proof}[\textbf{Proof}]
  We bound below by
  \[
  \sum_{\substack{q=\Sigma_n+1\\ q \in
      \mathcal{Q}(x,\psi)}}^{\Sigma_{n+1}} 1 \geq
  \Abs{\{qx\}_{q=\Sigma_n +1}^{\Sigma_{n+1}}\cap [0,\psi
    (\Sigma_{n+1}))} \geq
  \floor[\bigg]{\frac{\psi(\Sigma_{n+1})}{\ell_{\max}}}
  \]
  which by Lemma~\ref{lem:density} we can bound by
  \[
  \Abs{\{qx\}_{q=\Sigma_n +1}^{\Sigma_{n+1}}\cap [0,\psi
    (\Sigma_{n+1}))} \geq \floor[\bigg]{\psi (\Sigma_{n+1})\cdot
    \frac{L_{n+1}}{R}} \gg L_{n+1}\, \psi (\Sigma_{n+1})
  \]
  as $n\to\infty$, because we have assumed that $L_n\,\psi
  (\Sigma_n)\geq R$ eventually.
\end{proof}

The next lemma will allow us to assume without loss of generality that
$\psi$ satisfies the conditions of Lemma~\ref{lem:schmidt}.

\begin{lemma}\label{lem:wlog}
  Let $\{L_n\}$ be a sequence of $R$-bounded gap ratios for
  $x\in\mathbb{R}\backslash\mathbb{Q}$, and $\{\Sigma_n\}$ be the
  sequence defined by $\Sigma_n = L_1+L_2+\dots+L_n$. Let $\psi$ be an
  approximating function. There is an approximating function
  $\tilde\psi \geq \psi$ such that $L_n\,\tilde\psi (\Sigma_n)\geq R$
  and such that
  \[
  \sum_{q \in \mathcal{Q}(x, \tilde\psi)} \tilde\psi(q)^{d-1} = \infty
  \implies \sum_{q \in \mathcal{Q}(x, \psi)} \psi(q)^{d-1} = \infty
  \]
  for any $d\geq 3$.
\end{lemma}

\begin{proof}[\textbf{Proof}]
  Let $\varphi$ be the approximating function defined by $\varphi(q) =
  RL_n^{-1}$ where $q \in \left(\Sigma_{n-1}, \Sigma_n\right]$ and
  define $\tilde\psi(q) := \max\{\psi(q), \varphi(q)\}$. Let $A = \{q
  : \psi(q)\geq \varphi(q)\}$ and $B = \mathbb{N}\backslash A$. Then
  \begin{equation}\label{eqn:wlog}
    \sum_{q \in \mathcal{Q}(x, \tilde\psi)} \tilde\psi(q)^{d-1} = \sum_{q \in A\cap \mathcal{Q}(x, \psi)} \psi(q)^{d-1} + \sum_{q \in B\cap \mathcal{Q}(x, \varphi)} \varphi(q)^{d-1}.
  \end{equation}
  The second sum is bounded by
  \begin{equation}\label{eqn:wlog2}
    \sum_{q \in \mathcal{Q}(x, \varphi)} \varphi(q)^{d-1} = \sum_{n\in\mathbb{N}} \Abs{\mathcal{Q}(x,\varphi)\cap\left(\Sigma_{n-1}, \Sigma_n\right]}\left(\frac{R}{L_n}\right)^{d-1}.
  \end{equation}
  By Lemma~\ref{lem:density},
  \[
  \Abs{\mathcal{Q}(x, \varphi)\cap \left(\Sigma_{n-1},
      \Sigma_n\right]} < \frac{2R}{L_n} \div \ell_{\min} <
  \frac{2R}{L_n} \div \frac{1}{L_n} = 2R
  \]
  so we can bound~\eqref{eqn:wlog2} by
  \[
  2R^d \, \sum_{n\in\mathbb{N}} \left(\frac{1}{L_n}\right)^{d-1}
  \]
  which converges as long as $d-1>1$. Now~\eqref{eqn:wlog} shows that
  if $\sum_{q \in \mathcal{Q}(x, \tilde\psi)} \tilde\psi(q)^{d-1}$
  diverges, then so does $\sum_{q \in \mathcal{Q}(x, \psi)}
  \psi(q)^{d-1}$.
\end{proof}

\begin{remark*}
  Besides our repeated applications of Gallagher's Theorem,
  Lemma~\ref{lem:wlog} is the only other place where we need $d\geq
  3$. Notice that the sum $\sum_{n\in\mathbb{N}} L_n^{-1}$ can
  diverge, for example, if $x$ is badly approximable.
\end{remark*}

The following lemma should be compared with Lemma~\ref{lem:schmidt}.
\begin{lemma}\label{lem:antischmidt}
  Let $\{L_n\}$ be a sequence of $R$-bounded gap ratios for
  $x\in\mathbb{R}\backslash\mathbb{Q}$, and $\{\Sigma_n\}$ be the
  sequence defined by $\Sigma_n = L_1+L_2+\dots+L_n$. If $\psi$ is an
  approximating function, then
  \[
  \sum_{\substack{q=\Sigma_{n-1}\\ q \in
      \mathcal{Q}(x,\psi)}}^{\Sigma_n-1} 1 \ll L_n\, \psi
  (\Sigma_{n-1})
  \]
  as $n\to\infty$.
\end{lemma}

\begin{proof}[\textbf{Proof}]
  We bound above by
  \[
  \sum_{\substack{q=\Sigma_{n-1}\\ q \in
      \mathcal{Q}(x,\psi)}}^{\Sigma_n-1} 1 \leq
  \Abs{\{qx\}_{q=\Sigma_{n -1}}^{\Sigma_n-1}\cap [0,\psi
    (\Sigma_{n-1}))}
  \]
  which by Lemma~\ref{lem:density} we can bound by
  \[
  \Abs{\{qx\}_{q=\Sigma_{n-1}}^{\Sigma_n-1}\cap [0,\psi
    (\Sigma_{n-1}))} \leq \psi (\Sigma_{n-1})\div \frac{1}{RL_n} \ll
  L_n\, \psi (\Sigma_{n-1})
  \]
  as $n\to\infty$.
\end{proof}

\begin{lemma}\label{lem:wlogconv}
  If $\psi$ is an approximating function such that
  $\sum_{q\in\mathbb{N}}\psi(q)^d$ converges, then $\psi(q)\ll
  q^{-1/d}$.
\end{lemma}

\begin{proof}[\textbf{Proof}]
  Since $\psi$ is non-increasing, convergence of
  $\sum_{q\in\mathbb{N}}\psi(q)^d$ is equivalent to convergence of
  $\sum_{k\in\mathbb{N}} 2^k\, \psi(2^k)^d$, therefore we know that
  the terms $2^k\,\psi(2^k)^d$ approach $0$, meaning that for any
  $c>0$, we eventually have $\psi(2^k)^d < c\cdot 2^{-k}$. So we
  certainly satisfy $\psi(q)\ll q^{-1/d}$ on the sequence
  $\{2^k\}_{k\in\mathbb{N}}$ with some implied constant
  $C>0$. Everywhere else, we observe that every $q$ is between some
  $2^k$ and the next one, so
  \[
  2^k < q \leq 2^{k+1} \quad \textrm{and} \quad \psi(2^{k+1}) \leq
  \psi(q) <\psi(2^k).
  \]
  Combining these and our previous observations we find
  \[
  \psi(q)^d < \psi(2^k)^d\leq C \cdot2^{-k} \leq C \cdot\frac{2}{q},
  \]
  and we have shown $\psi(q)\ll q^{-1/d}$ with implied constant
  $(2C)^{1/d}$.
\end{proof}

% ===========================================
% ===========================================

% ===========================================
% ===========================================

\section{Proofs of divergence results}\label{sec:proofsdiv}

In this section, we work with approximating functions $\psi$ with the
property that $\sum_{q\in\mathbb{N}}\psi(q)^d$ diverges. Our goal is
to determine when we can guarantee the divergence of
\begin{equation}\label{eqn:gallaghersum}
  \sum_{q\in\mathcal{Q}(x,\psi)}\psi(q)^{d-1}
\end{equation}
so that we can apply Gallagher's extension of Khintchine's Theorem to
the hyperplane passing through $x\in\mathbb{R}$.  To this end, let us
define the subset $A(x,R)\subseteq\mathbb{N}$ as the concatenation $
A(x,R)=\left\{A_1^{(x,R)}, A_2^{(x,R)}, \dots \right\} $ of blocks
\[
A_\ell^{(x,R)} = [\Sigma_{\al_\ell}, \Sigma_{\en_{\ell}}-1]\cap
\mathbb{N}.
\]
We prove the following lemma.

\begin{lemma}\label{lem:subseries}
  Let $x\in\mathbb{R}\backslash\mathbb{Q}$. If there exists a number
  $R\geq 1$ such that
  \begin{equation}\label{eqn:subseries}
    \sum_{q\in A(x,R)}\psi(q)^d
  \end{equation}
  diverges, then~\eqref{eqn:gallaghersum} diverges.
\end{lemma}

\begin{proof}[\textbf{Proof}]
  We write partial sums of~\eqref{eqn:gallaghersum} along
  $\{\Sigma_N\}$ as
  \begin{equation*}
    \sum_{\substack{q=1 \\ q \in \mathcal{Q}(x,\psi)}}^{\Sigma_N} \psi(q)^{d-1} = \sum_{n = 0}^{N-1}\sum_{\substack{q=\Sigma_n +1 \\ q \in \mathcal{Q}(x,\psi)}}^{\Sigma_{n+1}} \psi(q)^{d-1},
  \end{equation*}
  where $\Sigma_N = L_1+L_2+\dots+L_N$, and $\Sigma_0 = 0$. Since
  $\psi$ is non-increasing we can bound below by
  \[
  \geq \sum_{n = 0}^{N-1}
  \psi(\Sigma_{n+1})^{d-1}\sum_{\substack{q=\Sigma_n+1\\ q \in
      \mathcal{Q}(x,\psi)}}^{\Sigma_{n+1}} 1
  \]
  and Lemma~\ref{lem:wlog} allows us to assume without loss of
  generality that $L_n \psi(\Sigma_n)\geq 2+R$, so that we can apply
  Lemma~\ref{lem:schmidt} to bound by
  \begin{equation*}
    \gg \sum_{n = 1}^{N} L_n \,\psi(\Sigma_n)^d.
  \end{equation*}
  Re-writing along the subsequence $\{\en_m -1\}$,
  \[
  \sum_{n = 1}^{\en_m - 1} L_n \,\psi(\Sigma_n)^d = \sum_{\ell=0}^m
  \sum_{n = \al_\ell}^{\en_{\ell} -1} L_n \,\psi(\Sigma_n)^d +
  \sum_{\ell=0}^{m-1} L_{\en_\ell} \,\psi(\Sigma_{\en_\ell})^d,
  \]
  we can safely ignore the second sum because it converges as
  $m\to\infty$. Since $L_{n+1} = L_n+1$, except when $n = \en_\ell$,
  \[
  \sum_{\ell=0}^m \sum_{n = \al_\ell}^{\en_{\ell} -1} L_n
  \,\psi(\Sigma_n)^d \gg \sum_{\ell=0}^m \sum_{n =
    \al_\ell}^{\en_{\ell} -1} L_{n+1} \,\psi(\Sigma_n)^d
  \geq \sum_{\ell=0}^m \sum_{n = \al_\ell}^{\en_{\ell} -1}
  \sum_{q=\Sigma_n}^{\Sigma_{n+1}-1}\psi(q)^d = \sum_{\ell=0}^m
  \sum_{q=\Sigma_{\al_\ell}}^{\Sigma_{\en_{\ell}}-1}\psi(q)^d,
  \]
  and taking $m\to\infty$, we have bounded~\eqref{eqn:gallaghersum}
  below by~\eqref{eqn:subseries} which implies the result.
\end{proof}

The challenge now is to determine when we can find $R\geq 1$ such
that~\eqref{eqn:subseries} diverges.

\subsection{Proofs of Theorems~\ref{thm:divergence}
  and~\ref{thm:upper}}

Since $\sum_{q\in\mathbb{N}}\psi(q)^d$ diverges, it is sufficient to
find $A(x,R)$ with positive lower asymptotic density in $\mathbb{N}$.

\begin{lemma}\label{lem:lad}
  We have
  \[
  \operatorname{\underline{d}}(A(x,R))>0\qquad\iff\qquad\limsup_{m\to\infty}
  \frac{L_{\al_m}^R}{\Sigma_{\al_m}^R} < 1,
  \]
  that is, $A(x,R)$ has positive lower asymptotic density for some
  $R\geq 1$ if and only if $x\in\mathbb{R}\backslash\mathbb{Q}$ has
  the positive density property.
\end{lemma}

\begin{proof}[\textbf{Proof}]
  Since $A(x,R)$ is made up of blocks of consecutive integers, the
  lower asymptotic density is achieved by computing along the
  subsequence corresponding to the points just before the left
  end-points of each block. That is,
  \begin{multline*}
    \operatorname{\underline{d}}(A(x,R)) = \liminf_{m\to\infty} \frac{\sum_{\ell\leq m}\abs{A_\ell}}{\min A_{m+1} - 1} = \liminf_{m\to\infty} \frac{\sum_{\ell\leq m}(\Sigma_{\en_{\ell}}- \Sigma_{\al_\ell})}{\Sigma_{\al_{m + 1}}-1} \\
    = \liminf_{m\to\infty} \frac{\Sigma_{\al_{m+1}} - \sum_{\ell\leq m+1}(L_{\al_\ell})}{\Sigma_{\al_{m + 1}}-1} \\
    = \liminf_{m\to\infty} \frac{\Sigma_{\al_{m+1}} - (L_{\al_{m + 1}} + L_{\al_m}+ L_{\al_{m-1}} + \dots+ L_{\al_1} + L_{\al_0})}{\Sigma_{\al_m + 1}-1} \\
    = 1 - \limsup_{m\to\infty} \left(\frac{L_{\al_{m + 1}} +
        L_{\al_m}+ L_{\al_{m-1}} + \dots+ L_{\al_1} + L_{\al_0} -
        1}{\Sigma_{\al_{m + 1}}-1}\right)
  \end{multline*}
  and so $\operatorname{\underline{d}}(A(x,R))>0$ if and only if
  \[
  \limsup_{m\to\infty} \left(\frac{L_{\al_{m + 1}} + L_{\al_m}+
      L_{\al_{m-1}} + \dots+ L_{\al_1} + L_{\al_0}}{\Sigma_{\al_{m +
          1}}}\right) < 1,
  \]
  but
  \[
  \lim_{m\to\infty} \left(\frac{L_{\al_m}+ L_{\al_{m-1}} + \dots+
      L_{\al_1} + L_{\al_0}}{\Sigma_{\al_{m + 1}}}\right) = 0,
  \]
  so we have proved the claim.
\end{proof}

Since divergent series diverge along subseries of positive lower
asymptotic density, this lemma all but solves the problem for fibers
over points with the positive density property. The following lemma
shows that, at least for some approximating functions, one can deal
with fibers over base-points that do not have the positive density
property.

\begin{lemma}\label{lem:uad}
  For any $x\in\mathbb{R}\backslash\mathbb{Q}$ and $R\geq 1$ we have
  that $A(x,R)$ has positive upper asymptotic density. \end{lemma}

\begin{proof}[\textbf{Proof}]
  Since $A(x,R)$ is made up of blocks of consecutive integers, the
  upper asymptotic density is achieved by computing along the
  subsequence corresponding to the right end-points of each
  block. That is,
  \begin{multline*}
    \operatorname{\overline{d}}(A(x,R)) = \limsup_{m\to\infty} \frac{\sum_{\ell\leq m}\abs{A_\ell}}{\max A_m} = \limsup_{m\to\infty} \frac{\sum_{\ell\leq m}(\Sigma_{\en_{\ell}}- \Sigma_{\al_\ell})}{\Sigma_{\en_m}-1} \\
    = \limsup_{m\to\infty} \frac{\Sigma_{\en_m} - \sum_{\ell\leq m}(L_{\al_\ell})}{\Sigma_{\en_m}-1} \\
    = \limsup_{m\to\infty} \frac{\Sigma_{\en_m} - (L_{\al_m}+ L_{\al_{m-1}} + \dots+ L_{\al_1} + L_{\al_0})}{\Sigma_{\en_m}-1} \\
    = 1 - \liminf_{m\to\infty} \left(\frac{L_{\al_m}+ L_{\al_{m-1}} +
        \dots+ L_{\al_1} + L_{\al_0} - 1}{\Sigma_{\en_m}-1}\right)
  \end{multline*}
  and so $\operatorname{\overline{d}}(A(x,R))>0$ if and only if
  \[
  \liminf_{m\to\infty} \left(\frac{L_{\al_m}+ L_{\al_{m-1}} + \dots+
      L_{\en_1} + L_{\al_0}}{\Sigma_{\en_m}}\right) < 1,
  \]
  but this is always the case.
\end{proof}

We can now prove that almost every point on every fiber is
$\psi$-approximable, if $\psi$ happens to have the property that
$\sum_A \psi(q)^d$ diverges for every $A\subseteq\mathbb{N}$ with
positive upper asymptotic density.
\begin{proof}[\textbf{Proof of Theorem~\ref{thm:upper}}]
  Lemma~\ref{lem:uad} tells us that for any
  $x\in\mathbb{R}\backslash\mathbb{Q}$ and $R\geq 1$, the set $A(x,R)$
  has positive upper asymptotic density. By assumption,
  then,~\eqref{eqn:subseries} diverges. Therefore, by
  Lemma~\ref{lem:subseries}, the sum~\eqref{eqn:gallaghersum}
  diverges. Since $d-1\geq 2$, Gallagher's Theorem applies to the
  hyperplane $\{x\}\times\mathbb{R}^{d-1}$ and approximating function
  $\psi$.
\end{proof}

These density considerations only give \emph{sufficient} conditions
for divergence, and the following lemma serves to show that they are
not necessary.

\begin{lemma}\label{lem:construction}
  For any $R\geq 1$, there are uncountably many
  $x\in\mathbb{R}\backslash\mathbb{Q}$ of any given Diophantine type
  such that~\eqref{eqn:subseries} diverges.
\end{lemma}

\begin{proof}[\textbf{Proof}]
  We offer a construction. Fix $R\geq 1$. Let
  \[
  \Psi(m):= \sum_{\ell=0}^m
  \sum_{q=\Sigma_{\al_\ell}}^{\Sigma_{\en_{\ell}}-1}\psi(q)^d
  \]
  be a partial sum of~\eqref{eqn:subseries}. The sequence
  $\{\Psi(m)\}$ is increasing, and notice that we can make
  \[
  \Psi(m) - \Psi(m-1) =
  \sum_{q=\Sigma_{\al_m}}^{\Sigma_{\en_m}-1}\psi(q)^d
  \]
  as large as we wish by choosing $\Delta k_{m-1} - 1:=k_m - k_{m-1} -
  1$ arbitrarily large, so we can make $\Psi(m)\to\infty$ simply by
  prescribing $\{k_m\}$.

  To see that we can achieve any Diophantine type, we observe that at
  each step, after having chosen $k_m$ so that $\Psi(m) - \Psi(m-1)$
  has the desired size, we are free to choose $a_{k_m+1}$ without
  affecting $\Psi(m)$. Therefore we can ensure that any given
  $\sigma\in[1, \infty)$ is the infimum over $\tau\in\mathbb{R}$
  satisfying $q_{k_m+1}\ll q_{k_m}^\tau$ as $m\to\infty$.
\end{proof}

We are now prepared to prove the following theorem, from which
Theorem~\ref{thm:divergence} immediately follows.
\begin{theorem}\label{thm:gapsumdiv}
  Let $d\geq 3$. If $\psi$ is an approximating function such that the
  sum $\sum_{q\in\mathbb{N}} \psi(q)^d$ diverges, then
  \[
  \sum_{q\in \mathcal{Q}(x ,\psi)} \psi(q)^{d-1}=\infty
  \]
  for:
  \begin{itemize}
  \item[\rm\textbf{(a)}] Any $x \in\mathbb{Q}$.
  \item[\rm\textbf{(b)}] Any $x\in\mathbb{R}\backslash\mathbb{Q}$ with
    the positive density property.
  \item[\rm\textbf{(c)}] Uncountably many
    $x\in\mathbb{R}\backslash\mathbb{Q}$ of any given Diophantine
    type.
  \end{itemize}
\end{theorem}

\begin{proof}[\textbf{Proof}]
  We treat the different parts of the theorem separately.

  \textbf{Part~(a):} In the case of rational $x = a/b$, the set
  $\mathcal{Q}(x, \psi)$ contains the arithmetic sequence
  $\{kb\}_{k\in\mathbb{N}}$. Then
  \[
  \sum_{q\in \mathcal{Q}(x ,\psi)} \psi(q)^{d-1} \gg
  \sum_{k\in\mathbb{N}} b\, \psi(kb)^{d-1} \gg \sum_{q\in\mathbb{N}}
  \psi(q)^{d-1} =\infty.
  \]
  We have used here that $\psi$ is non-increasing. (Notice that this
  does not require $d\geq 3$.)

  \textbf{Part~(b):} If $x\in\mathbb{R}\backslash\mathbb{Q}$ has the
  positive density property, then Lemma~\ref{lem:lad} implies that
  there is some $R\geq 1$ such that $A(x, R)$ has positive lower
  asymptotic density. This implies that~\eqref{eqn:subseries}
  diverges, which by Lemma~\ref{lem:subseries} implies
  that~\eqref{eqn:gallaghersum} diverges.

  \textbf{Part~(c):} By Lemma~\ref{lem:construction}, there are
  uncountably many $x\in \mathbb{R}\backslash\mathbb{Q}$ of any given
  Diophantine type such that~\eqref{eqn:subseries} diverges, and again
  Lemma~\ref{lem:subseries} implies that~\eqref{eqn:gallaghersum}
  diverges.
\end{proof}

\begin{proof}[\textbf{Proof of Theorem~\ref{thm:divergence}}]
  Theorem~\ref{thm:divergence} is proved by applying Gallagher's
  Theorem to fibers over the base points in
  Theorem~\ref{thm:gapsumdiv}.
\end{proof}

\subsection{Proofs of Theorems~\ref{thm:prototype}
  and~\ref{thm:prototypes}}

By the discussion~\S\ref{sec:onproofs}, the proof of
Theorem~\ref{thm:prototype} reduces to the following lemma.

\begin{lemma}\label{lem:prototype}
  If $x\in \mathbb{R}\backslash\mathbb{Q}$ is not Liouville,
  then~\eqref{eqn:subseries} diverges for the approximating function
  $\psi(q) = (q\log q)^{-1/d}$.
\end{lemma}

\begin{proof}[\textbf{Proof}]
  If $x$ has the positive density property, then there is some $R\geq
  1$ for which $A(x,R)$ has positive lower asymptotic density, by
  Lemma~\ref{lem:lad}. This implies that~\eqref{eqn:subseries}
  diverges.

  On the other hand, if $x$ does not satisfy the positive density
  property, this means that
  \[
  \limsup_{m\to\infty}\frac{L_{\al_m}^R}{\Sigma_{\al_m}^R} = 1
  \]
  no matter which $R\geq 1$ we choose. Therefore, after fixing some
  $R\geq 1$, there is some sequence $\{m_j\}\subseteq\mathbb{N}$ where
  the limit superior is achieved, which means that on this sequence we
  have $L_{\al_{m_j}} \sim \Sigma_{\al_{m_j}}$. The partial sums
  of~\eqref{eqn:subseries} are then bounded by
  \[
  \sum_{q=\Sigma_{\al_m}}^{\Sigma_{\en_{m}}-1}\psi(q)^d \geq
  \int_{\Sigma_{\al_m}}^{\Sigma_{\en_{m}}}\frac{1}{q\log q} =
  \log\frac{\log \Sigma_{\en_{m}}}{\log \Sigma_{\al_m}},
  \]
  but we have
  \[
  \frac{\log \Sigma_{\en_{m_j}}}{\log \Sigma_{\al_{m_j}}} \sim
  \frac{\log \Sigma_{\en_{m_j}}}{\log L_{\al_{m_j}}}
  \overset{\textrm{Lem.~\ref{lem:BmRsum}}}{\gtrsim} \frac{\log
    q_{k_{m_j}-1}q_{k_{m_j -1}}}{\log q_{k_{m_j-1}+1}} \geq 1 +
  \frac{1}{\sigma} \quad\textrm{as}\quad j\to\infty,
  \]
  where $\sigma\in [1, \infty)$ is such that $x\notin\mathcal
  W_1(\sigma)$. This implies that there is some $\delta>0$ such that
  \[
  \sum_{q=\Sigma_{\al_m}}^{\Sigma_{\en_{m}}-1}\psi(q)^d \geq \delta
  \]
  infinitely often. (We can take any $\delta< \log \left(1 +
    \frac{1}{\sigma}\right)$.) Hence~\eqref{eqn:subseries} diverges.
\end{proof}

\begin{proof}[\textbf{Proof of Theorem~\ref{thm:prototype}}]
  We can now apply the divergence part of Gallagher's Theorem to any
  fiber over a non-Liouville base point. The fibers over Liouville
  base points are covered by Khintchine's transference principle,
  after the remark at the end of~\S\ref{sec:prototypes}.
\end{proof}

The next two lemmas combine to form Theorem~\ref{thm:prototypesgen},
which is more general than Theorem~\ref{thm:prototypes}.

\begin{lemma}\label{lem:fastk}
  If $x\in\mathbb{R}\backslash\mathbb{Q}$ is not Liouville and there
  is some $\eps>0$ and $R\geq 1$ such that $\frac{\log^{s-1} \Delta
    k_m}{\log^{s-1} k_m} \geq {1+\eps}$ on a sequence of $m$'s,
  then~\eqref{eqn:subseries} diverges for the approximating function
  $\psi_{s,d}$.
\end{lemma}

\begin{proof}[\textbf{Proof}]
  Comparing sums to integrals we have
  \[
  \sum_{q \in A_{m+1}} \psi_{s,d}(q)
  \geq \log\left(\frac{\log^s \Sigma_{\en_{m+1}}}{\log^s
      \Sigma_{\al_{m+1}}}\right)
  \]
  and we will show that this expression is bounded below by
  $\log(1+\eps)$ on the sequence where $\frac{\log^{s-1} \Delta
    k_m}{\log^{s-1} k_m} \geq {1+\eps}$.

  On this sequence, we have
  \begin{multline*}
    \frac{\log^s\Sigma_{\en_{m+1}}}{\log^s \Sigma_{\al_{m+1}}}
    \overset{\textrm{Lem.~\ref{lem:BmRsum}}}{\gtrsim} \frac{\log^s
      q_{k_{m+1}}^2}{\log^s \max\left\{q_{k_m}q_{k_m-1},
        q_{k_m+1}\right\}}
    \overset{\textrm{Lem.~\ref{lem:fibonacci}}}{\gtrsim} \frac{\log^s \left(F(\Delta k_m)\, q_{k_m+1}\right)}{\log^s q_{k_m+1}} \\
    \overset{\textrm{Lem.~\ref{lem:qbound}}}{\gtrsim} \frac{\log^s
      \left(F(\Delta k_m)\, q_{k_m+1}\right)}{\log^s
      \left(R+1\right)^{\sigma^m k_m} }
    % YOU ARE HERE
    \gtrsim \frac{\log^{s-1} \left(\Delta k_m + \log
        q_{k_m+1}\right)}{\log^{s-1} k_m} \gtrsim 1 + \eps.
  \end{multline*}
  Therefore~\eqref{eqn:subseries} diverges.
\end{proof}

\begin{lemma}\label{lem:prototypes}
  If $x\in\mathbb{R}\backslash\mathbb{Q}$ is not Liouville, has
  essential Diophantine type greater than $1$, and $\Delta k_m \leq^*
  k_m$ for some $R\geq 1$, then there is a positive lower asymptotic
  density sequence $\{\ell_j\}\subseteq\mathbb{N}$ on which the
  comparison
  \[
  \int_{\Sigma_{\al_\ell}}^{\Sigma_{\en_\ell}} \psi_{s,d}(t)^d\, dt
  \gg \psi_{s-2,d}(\ell)^d
  \]
  holds, where $\psi_{s,d}(q) = (q\log q \log^2 q \dots \log^s
  q)^{-1/d}$. Therefore,~\eqref{eqn:subseries} diverges.
\end{lemma}

\begin{proof}[\textbf{Proof}]
  Let $1<\tilde\sigma <\sigma < \infty$ be such that $x\in\mathcal
  W_1^{\mathrm{ess}}(\tilde\sigma)\backslash\mathcal W_1(\sigma)$. We
  first show that
  \begin{equation}\label{eqn:basecase}
    \int_{\Sigma_{\al_\ell}}^{\Sigma_{\en_\ell}} \psi_{1,d}(t)^d\, dt = \log \left(\frac{\log\Sigma_{\en_\ell}}{\log\Sigma_{\al_\ell}}\right) \gg 1
  \end{equation}
  holds on a sequence $\{\ell_j\}\subseteq\mathbb{N}$ of positive
  lower asymptotic density, by showing that there is some $\eps>0$
  such that $\frac{\log\Sigma_{\en_\ell}}{\log\Sigma_{\al_\ell}}
  \gtrsim 1+\eps$ on a sequence of $\ell$'s of positive lower
  asymptotic density.

  By Lemma~\ref{lem:BmRsum} we have the comparisons
  $\Sigma_{\en_{m+1}} \gg q_{k_{m+1}}q_{k_m}$ and $\Sigma_{\al_{m+1}}
  \ll q_{k_m}q_{k_m-1} + q_{k_m+1}$, and because $x\in\mathcal
  W_1^{\mathrm{ess}}(\tilde\sigma)$ there is a sequence
  $\{m_j\}\subseteq\mathbb{N}$ of positive lower asymptotic density
  such that $q_{k_{m_j}}^{\tilde\sigma}< q_{k_{m_j+1}}$. We now have
  \[
  \frac{\log\Sigma_{\en_{m_j+1}}}{\log\Sigma_{\al_{m_j+1}}} \gtrsim
  \frac{\log q_{k_{m_j+1}}q_{k_{m_j}}}{\log
    \left(q_{k_{m_j}}q_{k_{m_j}-1} + q_{k_{m_j}+1}\right)} \gtrsim
  \frac{\log q_{k_{m_j+1}}q_{k_{m_j}}}{\log
    \max\left\{q_{k_{m_j}}q_{k_{m_j}-1}, q_{k_{m_j}+1}\right\}}.
  \]
  Whenever $q_{k_{m_j}}q_{k_{m_j}-1} \leq q_{k_{m_j}+1}$, this becomes
  \begin{equation*}
    \frac{\log q_{k_{m_j+1}}q_{k_{m_j}}}{\log q_{k_{m_j}+1}} = \frac{\log q_{k_{m_j+1}}}{\log q_{k_{m_j}+1}} + \frac{\log q_{k_{m_j}}}{\log q_{k_{m_j}+1}} \geq 1 + \frac{1}{\sigma}.
  \end{equation*}
  And whenever $q_{k_{m_j}}q_{k_{m_j}-1} \geq q_{k_{m_j}+1}$, we get
  \begin{equation*}
    \frac{\log q_{k_{m_j+1}}q_{k_{m_j}}}{\log q_{k_{m_j}}q_{k_{m_j}-1}} > \frac{\log q_{k_{m_j}}^{1+\tilde\sigma}}{\log q_{k_{m_j}}^2} = \frac{1+\tilde\sigma}{2} > 1
  \end{equation*}
  because $\tilde\sigma >1$. The sequence $\{\ell_j\}$ in the previous
  paragraph is $\ell_j = m_j +1$, and we have proved
  \[
  \frac{\log\Sigma_{\en_{\ell_j}}}{\log\Sigma_{\al_{\ell_j}}} \gtrsim
  1+\eps
  \]
  with any fixed
  \[
  0<\eps < \min\left\{\frac{1}{\sigma},
    \frac{\tilde\sigma-1}{2}\right\},
  \]
  and this establishes the comparison~\eqref{eqn:basecase}.

  We now show that
  \begin{equation}\label{eqn:secondbase}
    \int_{\Sigma_{\al_{\ell_j}}}^{\Sigma_{\en_{\ell_j}}} \psi_{2,d}(t)^d\, dt \gg \psi_{0, d}(\ell_j).
  \end{equation}
  Evaluating the integral gives
  \[
  \int_{\Sigma_{\al_{\ell_j}}}^{\Sigma_{\en_{\ell_j}}}
  \psi_{2,d}(t)^d\, d = \log \left(\frac{\log\log
      \Sigma_{\en_{\ell_j}}}{\log\log\Sigma_{\al_{\ell_j}}}\right)
  \geq \log \left(1+
    \frac{\log(1+\eps)}{\log\log\Sigma_{\al_{\ell_j}}}\right).
  \]
  Lemma~\ref{lem:qbound} and the assumption that $\Delta k_m \leq^*
  k_m$ imply
  \[
  \log\log\Sigma_{\al_{\ell_j}} \ll \ell_j,
  \]
  and recalling the fact that $\log\left(1+t\right) \sim t$ as $t\to
  0$, we have~\eqref{eqn:secondbase}.

  In the general case, we claim that for all $s\in\mathbb{N}$,
  \[
  \frac{\log^s \Sigma_{\en_{\ell_j}}}{\log^s\Sigma_{\al_{\ell_j}}}
  \gtrsim 1+ \frac{\log(1+
    \eps)}{\log^s\Sigma_{\al_{\ell_j}}\log^{s-1}\Sigma_{\al_{\ell_j}}\dots
    \log^2\Sigma_{\al_{\ell_j}}} = 1+ C_s \psi_{s-2,d}(\ell_j)^d
  \]
  where $C_s>0$. We have already proved the base case. In the
  inductive step,
  \begin{align*}
    \frac{\log^s \Sigma_{\en_{\ell_j}}}{\log^s \Sigma_{\al_{\ell_j}}} &\gtrsim 1 + \frac{\log (1+ C_{s-1}\psi_{s-3,d}(\ell_j)^d)}{\log^s \Sigma_{\al_{\ell_j}}} \\
    &\sim 1 + \frac{C_{s-1}\psi_{s-3,d}(\ell_j)^d}{\log^s \Sigma_{\al_{\ell_j}}} \\
    &= 1 + C_s \psi_{s-2,d}(\ell_j)^d,
  \end{align*}
  proving the claim. Evaluating the intergral,
  \begin{multline*}
    \int_{\Sigma_{\al_{\ell_j}}}^{\Sigma_{\en_{\ell_j}}}
    \psi_{s,d}(t)^d\, d = \log \left(\frac{\log^s
        \Sigma_{\en_{\ell_j}}}{\log^s\Sigma_{\al_{\ell_j}}}\right) \\
    \gtrsim \log \left(1+ C_s \psi_{s-2,d}(\ell_j)^d\right) \sim C_s
    \psi_{s-2,d}(\ell_j)^d \gg \psi_{s-2,d}(\ell_j)^d,
  \end{multline*}
  we have proved the lemma.
\end{proof}

\begin{theorem}\label{thm:prototypesgen}
  Let $x\in\mathbb{R}\backslash\mathbb{Q}$ be non-Liouville.
  \begin{itemize}
  \item[\textbf{(a)}] If for some $\eps>0$ there is an
    $s\in\mathbb{N}$ such that $\frac{\log^{s-1}\Delta k_m}{\log^{s-1}
      k_m} \geq 1+\eps$ for infinitely many $m\in\mathbb{N}$; or,
  \item[\textbf{(b)}] If the essential Diophantine type of $x$ is
    greater than $1$ and $\Delta k_m \leq^* k_m$,
  \end{itemize}
  then
  \[
  m_{d-1}\left(\mathcal{W}_{d}(\psi_{s,d})\cap
    \left(\{x\}\times\mathbb{R}^{d-1}\right)\right)=\textsc{full}.
  \]
\end{theorem}

\begin{proof}[\textbf{Proof}]
  Part (a) follows from Lemma~\ref{lem:fastk} and part (b) follows
  from Lemma~\ref{lem:prototypes}, both after applying
  Lemma~\ref{lem:subseries} and Gallagher's Theorem.
\end{proof}

\begin{proof}[\textbf{Proof of Theorem~\ref{thm:prototypes}}]
  Any $x\in\mathbb{R}\backslash\mathbb{Q}$ that is not Liouville and
  has regular Diophantine type greater than $1$ satisfies part (b) of
  Theorem~\ref{thm:prototypesgen}. For any $x$ of Diophantine type
  greater than $d$, the theorem is proved by the remark on
  Khintchine's transference principle at the end
  of~\S\ref{sec:prototypes}.
\end{proof}

\subsection{Another point of view}

Before moving our attention to the convergence results, we would like
to offer another point of view of what we have done
in~\S\ref{sec:proofsdiv}.

Notice that power set $\mathcal{P}(\mathbb{N})$ surjects onto $[0, 1]$
by mapping a subset $A\subseteq\mathbb{N}$ to the binary expansion
$0.d_1 d_2 d_3\dots$, where $d_q = \mathbf{1}_{A}(q)$ is the indicator
of $A$. In fact, the set $\mathcal{P}^\infty(\mathbb{N})$ of
\emph{infinite} subsets of $\mathbb{N}$ can be identified with $(0,1]$
by considering only binary expansions with infinitely many $1$'s. With
this identification $\mathcal{P}^\infty(\mathbb{N})\cong (0,1]$ in
mind, let us denote
\[
\mathcal{C}(\psi) := \left\{ A\in\mathcal{P}^\infty (\mathbb{N}) :
  \sum_{q\in A}\psi(q) < \infty \right\} \subset (0,1]
\]
and
\[
\mathcal{D}(\psi) := \left\{ A\in\mathcal{P}^\infty (\mathbb{N}) :
  \sum_{q\in A}\psi(q) = \infty \right\} =
(0,1]\backslash\mathcal{C}(\psi)
\]
to be the sets of convergent and divergent subseries of
$\sum_{q\in\mathbb{N}}\psi(q)$, respectively. {\v S}al{\'a}t offers
the following theorem.

\begin{theorem}[\cite{Salat}]\label{thm:salat}
  If $\psi:\mathbb{N}\to\mathbb{R}$ is non-increasing and
  $\sum_{q\in\mathbb{N}} \psi(q)$ diverges, then
  $\mathcal{C}(\psi)\subset (0,1]$ has Hausdorff dimension $0$.
\end{theorem}

In~\S\ref{sec:proofsdiv} we have explicitly defined a map
$A:\mathbb{R}\backslash\mathbb{Q}\times\mathbb{N}\to (0,1]$ using our
bounded ratio sequences $\left\{L_n^R\right\}$, and we have spent our
effort showing that $A(x,R)\in\mathcal D (\psi^d)$ for as many $x$'s
as possible, where $\psi$ is some approximating function satisfying
$\sum_{q\in\mathbb{N}}\psi(q)^d = \infty$. Theorem~\ref{thm:salat}
says that this amounts to showing that the map $A$ takes values in a
set whose complement has Hausdorff dimension $0$.

It is tempting to hope that closer analysis of the properties of the
map $A$ will reveal that the preimage of $\mathcal C(\psi^d)$ must
also have Hausdorff dimension $0$. This would prove the following
statement:
\begin{quote}
  \emph{Let $d\geq 3$. If $\psi$ is an approximating function such
    that the sum $\sum_{q\in\mathbb{N}} \psi(q)^d$ diverges, then
    \[
    m_{d-1}\left(\mathcal{W}_{d}(\psi)\cap
      \left(\{x\}\times\mathbb{R}^{d-1}\right)\right)=\textsc{full}
    \]
    for all $x\in\mathbb{R}\backslash E$, where the (possibly empty)
    set $E$ of exceptions has Hausdorff dimension zero.}
\end{quote}
We can reasonably expect this to be true (even with an empty $E$). In
particular, we have already proved it for the prototypical $\psi(q) =
(q\log q)^{-1/d}$, in Theorem~\ref{thm:prototype}, and for
approximating functions with the property that any convergent
subseries of $\sum \psi(q)^d$ has asymptotic density zero, in
Theorem~\ref{thm:upper}.

\section{Proofs of convergence results}

\subsection{Proof of Theorem~\ref{thm:convergence}}
The following theorem is a counterpart to Theorem~\ref{thm:gapsumdiv},
and Theorem~\ref{thm:convergence} follows immediately.
\begin{theorem}\label{thm:gapsumconv}
  Let $d\geq 2$. If $\psi$ is an approximating function such that the
  sum $\sum_{q\in\mathbb{N}} \psi(q)^d$ converges, then
  \begin{equation*}
    \sum_{q\in \mathcal{Q}(x ,\psi)} \psi(q)^{d-1} < \infty
  \end{equation*}
  for:
  \begin{itemize}
  \item[\rm\textbf{(a)}] $\begin{dcases} \textrm{No $x \in\mathbb{Q}$}
      &\textrm{if $\sum_{q\in\mathbb{N}}\psi(q)^{d-1}$ diverges.} \\
      \textrm{Every $x\in\mathbb{R}$} &\textrm{if it
        converges.}\end{dcases}$
  \item[\rm\textbf{(b)}] Any $x\in\mathbb{R}\backslash\mathbb{Q}$ with
    the bounded ratio property.
  \end{itemize}
\end{theorem}

\begin{proof}[\textbf{Proof}]
  Again, we treat the different parts of the theorem separately.

  \textbf{Part~(a):} Suppose $\sum_{q\in\mathbb{N}}\psi(q)^{d-1}$
  diverges. Let $x = a/b$ be a rational number. Then the sequence
  $\{kb\}_{k\in\mathbb{N}}$ is contained in $\mathcal{Q}(x,\psi)$, so
  \[
  \sum_{q\in \mathcal{Q}(x ,\psi)} \psi(q)^{d-1} \geq
  \sum_{k\in\mathbb{N}} \psi(kb)^{d-1}
  \]
  and this diverges as in the proof of
  Theorem~\ref{thm:gapsumdiv}(a). On the other hand, if
  $\sum_{q\in\mathbb{N}}\psi(q)^{d-1}$ converges, then it is obvious
  that so does $\sum_{q\in \mathcal{Q}(x ,\psi)} \psi(q)^{d-1}$,
  regardless of whether $x$ is rational or irrational.

  \textbf{Part~(b):} We want to show that the sum $\sum_{q \in
    \mathcal{Q}(x,\psi)} \psi(q)^{d-1}$ converges. Similar to the
  proof of Theorem~\ref{thm:gapsumdiv}, we partition partial sums by
  \begin{equation}\label{eqn:tobound}
    \sum_{\substack{q=1 \\ q \in \mathcal{Q}(x,\psi)}}^{\Sigma_N - 1} \psi(q)^{d-1} = \sum_{n = 1}^{N}\sum_{\substack{q=\Sigma_{n-1} \\ q \in \mathcal{Q}(x,\psi)}}^{\Sigma_n-1} \psi(q)^{d-1},
  \end{equation}
  and we proceed to bound. First, we have
  \[
  \leq \sum_{n = 1}^{N} \psi(\Sigma_{n-1})^{d-1}
  \sum_{\substack{q=\Sigma_{n-1} \\ q \in
      \mathcal{Q}(x,\psi)}}^{\Sigma_n-1} 1 \ll \sum_{n = 1}^{N}
  L_n\,\psi(\Sigma_{n-1})^d,
  \]
  by Lemma~\ref{lem:antischmidt}. Now, as before, we have $L_n =
  L_{n-1} + 1$ except when $n = \al_m$ for some $m\geq 2$, so
  \[
  \ll \sum_{n = 1}^{N} L_{n-1}\,\psi(\Sigma_{n-1})^d +
  \sum_{n=1}^{N}\psi(\Sigma_{n-1})^d + \sum_{m\in\mathbb{N}}
  \left(B_{m+1} - B_m\right)\psi(\Sigma_{\en_m})^d
  \]
  and recalling Lemma~\ref{lem:wlogconv}
  \[
  \ll \sum_{n = 1}^{N} L_{n-1}\,\psi(\Sigma_{n-1})^d +
  \sum_{n=1}^{N}\psi(\Sigma_{n-1})^d +
  \sum_{m\in\mathbb{N}}\frac{B_{m+1} - B_m}{\Sigma_{\en_m}},
  \]
  which converges for some $R\geq 1$ if $x$ has the bounded ratio
  property. Therefore,~\eqref{eqn:tobound} converges as $N\to\infty$,
  as wanted.
\end{proof}

\begin{proof}[\textbf{Proof of Theorem~\ref{thm:convergence}}]
  This follows from Theorem~\ref{thm:gapsumconv} in the same way that
  Theorem~\ref{thm:divergence} follows from
  Theorem~\ref{thm:gapsumdiv}. This time, instead of applying
  Gallagher, we use the convergence part of Khintchine's Theorem (or,
  really, the Borel--Cantelli Lemma) to say that convergence of
  $\sum_{q\in\mathbb{N}} \bar\psi(q)^{d-1}$ implies that the measure
  of $\mathcal{W}_{d-1}(\bar\psi)$ is zero.
\end{proof}

\subsection{Proof Corollaries~\ref{cor},~\ref{cor:phi},
  and~\ref{cor:2}}
\begin{proof}[\textbf{Proof of Corollaries~\ref{cor},~\ref{cor:phi}
    and~\ref{cor:2}}]
  On the other hand, if $\psi(q)\leq q^{-(1+\delta)/d}$ for some
  $\delta>0$, then we can bound by
  \[
  \ll \sum_{n = 1}^{N} L_{n-1}\,\psi(\Sigma_{n-1})^d +
  \sum_{n=1}^{N}\psi(\Sigma_{n-1})^d +
  \sum_{m\in\mathbb{N}}\frac{B_{m+1} -
    B_m}{\Sigma_{\en_m}^{1+\delta}}.
  \]
  The first two terms converge, so let us look at the last. Its
  convergence is equivalent to that of
  \begin{equation}\label{eqn:corsum}
    \sum_{m\in\mathbb{N}}\frac{L_{\al_{m+1}}}{\Sigma_{\en_m}^{1+\delta}}
  \end{equation}
  so in particular, this converges if $x$ has the bounded ratio
  property, which proves Corollary~\ref{cor}.

  But by Lemma~\ref{lem:BmRsum} we can compare the summand
  $\frac{L_{\al_{m+1}}}{\Sigma_{\en_m}^{1+\delta}}$ to ratios of
  continuants. An argument almost identical to that of
  Proposition~\ref{prop:brp} will show that~\eqref{eqn:corsum}
  converges if $x$ meets the same restrictions on Diophantine type as
  in that proposition. The only difference is that now we have taken
  the denominators in the calculations to the power $1+\delta$, which
  allows our restrictions on Diophantine type to be inclusive, rather
  than exclusive.
\end{proof}

% ===========================================
% ===========================================
% ===========================================
% ===========================================
% ===========================================

\subsection*{Acknowledgments}

The author thanks Victor Beresnevich and Sanju Velani for suggesting
this problem and for many educational conversations, David Simmons for
alerting him to Khintchine's transference principle, and the referee
for making suggestions that improved the presentation.

% ===========================================
% ===========================================
% ===========================================
% ===========================================
% ===========================================

% \bibliographystyle{amsalpha} \bibliography{../bibliography}

\end{document}